\documentclass{article}
\usepackage{hyperref}

\usepackage{a4wide}
\usepackage{graphicx}
\usepackage{amsmath}
\usepackage{amssymb}
\usepackage{amsthm}
\usepackage{enumerate}
\usepackage{color}
\usepackage{transparent}

\newtheorem{theorem}{Theorem}[section]
\newtheorem{definition}[theorem]{Definition}

\newtheorem{lemma}[theorem]{Lemma}
\newtheorem{claim}[theorem]{Claim}
\newtheorem{corollary}[theorem]{Corollary}
\newtheorem{conjecture}[theorem]{Conjecture}

\newtheorem{observation}[theorem]{Observation}

\begin{document}
\title{Ramsey goodness of cycles}

\date{}

\author{Alexey Pokrovskiy\thanks{Department of Mathematics, ETH, 8092 Zurich, 
Switzerland. {\tt dr.alexey.pokrovskiy@gmail.com}. 
Research supported in part by SNSF grant 200021-175573.}
\and
Benny Sudakov
\thanks{Department of Mathematics, ETH, 8092 Zurich, Switzerland. 
{\tt benjamin.sudakov@math.ethz.ch}. 
Research supported in part by SNSF grant 200021-175573.}}
\maketitle
\begin{abstract}
Given a pair of graphs $G$ and $H$, the Ramsey number $R(G,H)$ is the smallest $N$ such that every red-blue coloring of the edges of the complete graph $K_N$ contains a red copy of $G$ or a blue copy of $H$. If a graph $G$ is connected, it is well known and easy to show that $R(G,H) \geq (|G|-1)(\chi(H)-1)+\sigma(H)$, where $\chi(H)$ is the chromatic number of $H$ and $\sigma(H)$ is the size of the smallest color class in a $\chi(H)$-coloring of $H$. A  graph $G$ is called $H$-good if $R(G,H)= (|G|-1)(\chi(H)-1)+\sigma(H)$. The notion of Ramsey goodness was introduced by Burr and Erd\H{o}s in 1983 and has been extensively studied since then.

In this paper we show that  if $n\geq 10^{60}|H|$ and $\sigma(H)\geq \chi(H)^{22}$ then the $n$-vertex cycle $C_n$ is $H$-good. For graphs $H$ with high $\chi(H)$ and $\sigma(H)$, this proves in a strong form a conjecture of Allen, Brightwell, and Skokan.
\end{abstract}

\section{Introduction}
A celebrated theorem of Ramsey from 1930 says that for every $n$, there is a number $R(n)$ such that any $2$-edge-colouring of a complete graph on $R(n)$ vertices contains a monochromatic complete subgraph on $n$ vertices.  Estimating $R(n)$ is a very difficult problem and one of the central problems in combinatorics. For a pair of graphs $G$ and $H$, we can define the \emph{Ramsey number $R(G,H)$}, to be the smallest integer $N$ such that any red-blue edge coloring of the complete graph on $N$ vertices contains a red copy of $G$ or a blue copy of $H$. 
As a corollary of Ramsey's Theorem, $R(G,H)$ is finite, since we always have $R(G,H)\leq R(\max(|G|, |H|))$. 

Although in general determining $R(G,H)$ is very difficult, for some pairs of graphs $G$ and $H$, their Ramsey number can be computed exactly. For example, Erd\H{o}s~\cite{ErdosRamsey} in 1947 showed that the Ramsey number of an $n$-vertex path versus a complete graph of order $m$ satisfies $R(P_n, K_m)=(n-1)(m-1)+1$. The construction showing that this is tight comes from considering a $2$-edge-colouring of $K_N, N=(n-1)(m-1)$ consisting of $m-1$ disjoint red cliques of size $n-1$ with all the edges between them blue. It is easy to check that this colouring has no red $P_n$ or blue $K_m$. Chv\'atal and Harary observed that the same construction serves as a lower bound for $R(G, H)$ where $G$ is any connected graph on $n$ vertices and $H$ is an $m$-partite graph.
Let $\chi(H)$ be the {\it chromatic number} of $H$, i.e. the smallest number of colors needed to color the vertices of $H$ so that no pair of adjacent vertices have the same colour, and $\sigma(H)$ be the the size of the smallest color class in a $\chi(H)$-colouring of $H$.
Refining the above construction, Burr~\cite{Bu} obtained the following lower bound for the Ramsey number of a pair of graphs.
\begin{lemma}[Burr,~\cite{Bu}]\label{LemmaRamseyLowerBound}
Let $H$ be a graph, and $G$ a connected graph with  $|G|\geq\sigma(H)$, we have
\begin{equation}\label{RamseyLowerBound}
R(G,H)\geq  (|G|-1)(\chi(H)-1)+\sigma(H).
\end{equation}
\end{lemma}
To prove this bound, consider a $2$-edge-coloring of complete graph on  $N=(|G|-1)(\chi(H)-1)+\sigma(H)-1$ vertices consisting of $\chi(H)-1$ disjoint red cliques of size $|G|-1$ as well as one disjoint red clique of size $\sigma(H)-1$. This coloring has no red $G$ because all red connected components have size $\leq |G|-1$, and there is no blue $H$ since the partition of this $H$ induced by red cliques would give a coloring of $H$ by $\chi(H)$ colors with one color class smaller than $\sigma(H)$, contradicting the definition of $\sigma(H)$.

The bound in Lemma~\ref{LemmaRamseyLowerBound} is very general, but for some graphs is quite far from the truth.  For example Erd\H{o}s~\cite{ErdosRamsey} showed that $R(K_n, K_n)\geq \Omega(2^{n/2})$ which is much larger than the quadratic bound we get from (\ref{RamseyLowerBound}). However there are many known pairs of graphs (such as when $G$ is a path and $H$ is a clique) for which $R(G,H)=(|G|-1)(\chi(H)-1)+\sigma(H)$. If this is a case we say that \emph{$G$ is $H$-good}. The notion of Ramsey goodness was introduced by Burr and Erd\H{o}s \cite{BE} in 1983 and was extensively studied since then.

A lot of early research on Ramsey-goodness focused on proving that particular pair of graphs is good. For example 
Gerencser and Gy\'arf\'as~\cite{GG} showed that for $n\geq m$ the path $P_n$ is $P_m$-good. Chv\'atal showed that any tree $T$ is $K_m$-good~\cite{Ch}. For more recent progress on Ramsey-goodness see~\cite{ABS, CFLS, FGMSS,Nik, NR} and their references.

The problem of Ramsey-goodness of cycles goes back to the work of Bondy and Erd\H{o}s \cite{BonE}, who proved that the 
cycle $C_n$ is $K_m$-good when $n \geq m^2-2$. Motivated by their result, Erd\H{o}s, Faudree, Rousseau, and Schelp conjectured that:
\begin{conjecture}[Erd\H{o}s, Faudree, Rousseau, and Schelp~\cite{EFRSCycleComplete}]\label{ConjectureCycleComplete}
If $n\geq m\geq 3$ then $R(C_n, K_m)=(n-1)(m-1)+1$.
\end{conjecture}

\noindent
Over the years, this problem has attracted a lot of attention. After several improvements, the best current result is due to Nikoforov ~\cite{Nik}, who showed that conjecture holds for $n\geq 4m+2$. In addition several authors proved it for small $m$ (see~\cite{CCZ} and the references therein).

In this paper we investigate Ramsey-goodness of an $n$-vertex cycle versus a general graph $H$. When $n$ is sufficiently large as a function of 
$|H|$, Burr and Erd\H{o}s ~\cite{BE} proved more than 30 years ago that $C_n$ is $H$-good. Recently Allen, Brightwell, and Skokan conjectured that the cycle is $H$-good already when its length is linear in the order of $H$.
\begin{conjecture}[Allen, Brightwell, and Skokan~\cite{ABS}]\label{ConjectureCnH}
For $n\geq \chi(H) |H|$ we have $R(C_n, H)=(n-1)(\chi(H)-1)+\sigma(H)$.
\end{conjecture}

\noindent
There have been some work (see, e.g., \cite{PeiLi} and it references) showing that the path $P_n$ is $H$-good. Since $R(P_n,H)$ is always at most $R(C_n, H)$, a weakening of the above conjecture is to show that $P_n$ is $H$-good for $n\geq \chi(H)|H|$. This was achieved by the authors of this paper in \cite{PSPaths}.

In this paper, we prove the following result.
\begin{theorem}\label{TheoremCnKmkRamsey}
For $n \geq 10^{60}m_k$ and $m_k\geq m_{k-1}\geq \dots \geq m_1$ satisfying $m_i\geq i^{22}$, we have $R(C_n, K_{m_1, \dots, m_k}) = (n-1)(k-1)+m_1.$
\end{theorem}

\noindent
Here $K_{m_1, \ldots, m_k}$ is a complete multipartite with $k$  parts of sizes $m_1, \dots, m_k$.
Notice that the vertices of a $k$-chromatic graph  $H$ can be partitioned into $k$ independent sets of sizes $m_1, \dots, m_k$ with  
$\sigma(H)=m_1\leq m_2\leq \dots \leq m_k$.  This is equivalent to $H$ being a subgraph of $K_{m_1, \dots, m_k}$.
Therefore Theorem~\ref{TheoremCnKmkRamsey} implies the following.

\begin{corollary}\label{RamseyCycleCorollary}
Suppose that we have numbers $n$,  and a graph $H$ with $n\geq 10^{60} |H|$ and $\sigma(H)\geq \chi(H)^{22}$. Then
$R(C_n, H)=(n-1)(\chi(H)-1)+\sigma(H)$.
\end{corollary}

For graphs  $H$ with large $\chi(H)$ and $\sigma(H)$, the above theorem proves Conjecture~\ref{ConjectureCnH} in a very strong form---it shows that in this case, the condition ``$n\geq \chi(H)|H|$'' is unnecessary, and $n\geq 10^{60}|H|$ suffices. 
For certain graphs $H$, Theorem~\ref{TheoremCnKmkRamsey} shows that $C_n$ is $H$-good in a range which is even better than ``$n\geq 10^{60}|H|$''. For example if $H$ is balanced (i.e. if $|H|=\sigma(H)\chi(H)$), then  Theorem~\ref{TheoremCnKmkRamsey} implies that $C_n$ is $H$-good as long as $n\geq 10^{60}|H|/\chi(H)$.

\subsection{Proof sketch}
Here we give an informal sketch of the proof of Theorem~\ref{TheoremCnKmkRamsey}. For simplicity we talk just about the balanced case of the theorem i.e. the proof of $R(C_n, K_m^k)=(k-1)(n-1)+m$.

Let $R(C_{\geq n}, K_{m_1, \dots, m_k})$ denote the smallest number $N$ such that in every colouring of $K_N$ by the colours red and blue there is a red cycle of length \emph{at least $n$} or a blue $K_{m_1, \dots, m_k}$. In~\cite{PSPaths} the following theorem is proved.
\begin{theorem}\label{TheoremRamseyAtLeast}
Given integers $m_1\leq m_2\leq \dots \leq m_k$ and $n\geq 3m_k+5m_{k-1}$, we have
$$R(C_{\geq n}, K_{m_1, \dots, m_k})= (k-1)(n-1)+m_1.$$
\end{theorem}
Notice that the above theorem is essentially a version of Theorem~\ref{TheoremCnKmkRamsey}, except  that it produces a red cycle of length \emph{at least $n$} rather than one of length \emph{exactly $n$}. The proof of our main theorem uses many ideas from the proof of Theorem~\ref{TheoremRamseyAtLeast}. Because of this it may help readers to familiarize themselves with the very short proof of that theorem in~\cite{PSPaths}. It can be summarized as follows: If $K_N$ is coloured so that there is no blue $K_{m_1, \dots, m_k}$, then we use induction to find a large red subgraph $G$ in $K_N$ which is an expander. Then we use the famous P\'osa rotation-extension technique to find a long red cycle in $G$.

To prove Theorem~\ref{TheoremCnKmkRamsey} we use a similar strategy, except that we build a red cycle of length at least $n$ to also contain a special red subgraph called a \emph{gadget}. Informally a gadget is a path between two special vertices $x$ and $y$ which has many chords. Because of these chords, the gadget has the property that it has paths between $x$ and $y$ of many different lengths. A consequence of this is that if we can find a  cycle $C$ of length \emph{at least $n$} which contains a suitable gadget, then $C$ also contains a cycle of length \emph{exactly $n$}. Thus the proof of Theorem~\ref{TheoremCnKmkRamsey} naturally splits into two parts. The first part is to show that a large graph with no blue $K_{m}^k$ contains a gadget (see Section~\ref{SectionGadgets}). The second part is to build a cycle of length at least $n$ containing a gadget we found (see Section~\ref{SectionRamsey}).

To find a gadget in a graph with no blue $K_m^k$, we make heavy use of expanders. It turns out that if $K_N$ has no blue $K_m^k$, then it contains a large red subgraph $G$ with good expansion properties (see Lemma~\ref{LemmaExpanderExistenceMultipartite}). Once we have an expander, we prove several lemmas which find various structures inside expanders such as trees (Lemma~\ref{LemmaEmbedTrees}), paths~(Lemma~\ref{LemmaExpanderPathsPrescribedVertices}), and cycles~(Lemma~\ref{LemmaEmbedLongCycle}). We then put these structures together to build a gadget (Lemma~\ref{LemmaGadgetExistence}).
We remark that the gadgets that we use are very similar to \emph{absorbers} introduced by Montgomery in~\cite{Montgomery} during the study of spanning trees in random graphs.

After constructing gadgets, the proof of Theorem~\ref{TheoremCnKmkRamsey} has three main ingredients---Lemmas~\ref{LemmaRamseyConnectGivenVertices}, \ref{LemmaConnectedRamsey}, and \ref{LemmaPartition}. 

The first ingredient, Lemma~\ref{LemmaRamseyConnectGivenVertices}, should be thought of as a version of the $k=2$ case of Theorem~\ref{TheoremCnKmkRamsey}. Since the full proof of Theorem~\ref{TheoremCnKmkRamsey} is inductive, Lemma~\ref{LemmaRamseyConnectGivenVertices} serves as the initial case of the induction. The proof of this lemma is quite similar to the proof of Theorem~\ref{TheoremRamseyAtLeast} in~\cite{PSPaths}, with one extra ingredient---namely gadgets.

The second ingredient, Lemma~\ref{LemmaConnectedRamsey}, should be thought of as a strengthening of Theorem~\ref{TheoremCnKmkRamsey} in the case when the red subgraph of $K_N$ is highly connected. In this case it turns out that the Ramsey number can be lowered significantly (to $n+0.07kn$.) The proof of this lemma again uses gadgets.

The third ingredient, Lemma~\ref{LemmaPartition}, should be thought of as a stability version of Theorem~\ref{TheoremCnKmkRamsey}. It says that for $N$ close to $R(C_n, K_{m}^k)$, if we have a $2$-coloured $K_N$ with no red $C_n$ or blue $K_m^k$, then the colouring on $K_N$ must be close to the extremal colouring. Specifically it shows that most of the graph can be partitioned into $k-1$ large sets $A_1, \dots, A_{k-1}$ with only blue edges between them. 
Once we have this structure, Theorem~\ref{TheoremCnKmkRamsey} is fairly easy to prove---since  $A_1, \dots, A_{k-1}$ only have blue edges between them, they cannot contain a blue $K_{m}^2$ (or else the whole graph would contain a blue $K_m^k$.) Then we apply the a version of the $k=2$ case of Theorem~\ref{TheoremCnKmkRamsey} to one of the sets $A_i$ to obtain a red $C_n$ (specifically we apply Lemma~\ref{LemmaRamseyConnectGivenVertices} which serves as the ``initial case''  of the induction.)

\subsection{Notation}
Throughout this paper the \emph{order} of a path $P$, denoted $|P|$ is the number of vertices it has. The \emph{length} of $P$ is the number of edges $P$ has, which is $|P|-1$. Similarly, for a cycle $C$, both the order and length of $C$ are defined to be $|C|$, the number of vertices of $C$.
If $P=p_1, p_2, \dots, p_t$ is a path, then $p_1$ and $p_t$ are called the \emph{endpoints} of $P$, and $p_2, \dots, p_{t-1}$ are called the \emph{internal vertices} of $P$. We will say things like ``$P$ is internally contained in $S$'' or ``$P$ is internally disjoint from $S$'' to mean that the internal vertices of $P$ are contained in $S$ or disjoint from $S$.
For a graph $G$ and two vertices $x,y\in G$ we let $d_G(x,y)$ be the length of the shortest path in $G$ between $x$ and $y$.

Recall that a forest is a graph with no cycles, and a tree is a connected graph with no cycles. A \emph{rooted tree} is a tree with a designated vertices called the root. In tree $T$ with root $r$, we call $T\setminus \{r\}$ the \emph{internal} vertices of T. We think of the edges in a rooted tree as being directed away from the root. Then for a vertex $v$, the out-neighbours of $v$ are called the \emph{children} of $v$, and the in-neighbour of $v$ is the \emph{parent} of $v$. The \emph{depth} of a rooted tree is the maximum distance of a vertex from the root. A binary tree is a tree of maximum degree $3$. Notice that for any $m$, there is a rooted binary tree of depth  $\lceil \log m\rceil$ and order $m$.

Recall that for a vertex $v$ in a graph $G$ $N_G(v)$ denotes the neighbourhood of $v$ in $G$---the set of vertices with edges going to $v$.
For a set of vertices $S$ in a graph $G$ we let $N_G(S)=\bigcup_{s\in S} N_G(S)$ denote the set of neighbours in $G$ of vertices of $S$. For $U\subseteq G$, we let $N_U(S)=N_G(S)\cap U= \{u\in U: us\text{ is an edge for some }s\in S\}$. When there is no ambiguity in what the underlying graph is, we will abbreviate $N_G(S)$ to $N(S)$.

The \emph{complement} of a graph $G$, denoted $\overline{G}$, is the graph on $V(G)$ with $xy\in E(\overline G)\iff xy\not\in G$.
Notice that $R(H,K)\leq R$ is equivalent to saying that in any graph on $G$ on $R$ vertices either $G$ contains $H$ or $\overline G$ contains $K$.  
We let $K_m^k$ denote the complete multipartite graph with $k$ parts of size $m$. With this notation, $K_m^1$ means a set of $k$ vertices (with no edges.) Notice that we have $R(K_m^1,G)\leq m$ for any graph $G$.

Throughout the paper ``$\log$'' always means ``$\log_2$'', the binary log.
In this paper we will omit floor and ceiling signs where they are not essential.

\section{Gadgets}\label{SectionGadgets}
In this section we construct gadgets which are one of the main technical tools which we use in this paper. A gadget is a graph containing paths of several different lengths between a designated pair of vertices $a$ and $b$. 
\begin{definition}
A $k$-gadget  is a graph $J$ containing two vertices $a$ and $b$ such that $J$ has $ a $ to $ b $ paths of orders $ | J | $ and $ | J | - k $. 
\end{definition}
The vertices $a$ and $b$ are called the \emph{endpoints}  of the $k$-gadget. 
We will often identify a $k$-gadget $J$ with the path of order $|J|$ contained in it. 
A $(\leq k)$-gadget is a graph $J$ with two vertices $a$ and $b$ with $ a $ to $ b $ paths of lengths $ |J| $, $|J| -1 $, $ \dots $,  $ |J| - k $. In other words a $(\leq k)$-gadget is simultaneously a $k'$-gadget for $k' = 1, 2, \dots k $.

An example of a $k$ gadget is a cycle   with $k+2$ vertices with $a$ and $b$ a pair of adjacent vertices. Then $a$ to $b$ paths of orders $k+2$ and $2$ can be obtained by going around the cycle in different directions. For our purposed we will construct more complicated gadgets. The reason for this is that short cycles do not necessarily exist in graphs whose complements are $K_m^k$-free.

The main goal of this section is to prove the following lemma.
\begin{lemma}\label{LemmaGadgetExistence}
There exists a constant $N_1=10^7$ so that the following holds for any $\lambda, \mu, k$,$m \in \mathbb N$ with $m\geq k^3$, $\lambda\geq 2\mu\geq 10^9$, and $\mu m \geq 4100 (\lambda m)^{\frac 34}$.

Let $G$ be a graph with $|G|\geq (N_1\lambda \mu k)m$ and with $\overline G$ $K_m^k$-free. Then $G$ contains a $(\leq \lambda  m)$-gadget $J$ of order $(\lambda +\mu)m$ with endpoints $a$ and $b$ as well as an internally disjoint $a$ -- $b$ path $Q$ of order $\mu m$.
\end{lemma}
The above lemma could be be rephrased as a Ramsey-type statement. If we let $\mathcal J_{t, n}$ be the family of all $\leq t$ gadgets on $n$ vertices, then Lemma~\ref{LemmaGadgetExistence} implies that $R(\mathcal{J}_{\mu m, (\lambda+\mu )m}, K_{m}^k)\leq (N_1\lambda \mu k)m$.

Notice that Lemma~\ref{LemmaGadgetExistence} also finds a path $Q$ between the two endpoints of the gadget it produces. This path should be thought of as a technical tool which we will later use to join gadgets together.

The structure of this section is as follows. In Section~\ref{SectionExpandersGadgets} we introduce expanders and give their basic properties. In Section~\ref{SectionTrees} we give a variant of a result of Friedman and Pippenger about embedding trees into expanders. In Section~\ref{SectionPathsAndCycles} we prove some lemmas about embedding paths and cycles into expanders. In Section~\ref{SectionConstructingGadgets} we prove Lemma~\ref{LemmaGadgetExistence}. In Section~\ref{SectionGadgetCycles} we prove some additional properties of gadgets which we will need.

\subsection{Expanders}~\label{SectionExpandersGadgets}
We'll use the following notion of expansion.
\begin{definition}
For a graph $G$ and $W\subseteq V(G)$, we say that $G$ $(\Delta, \beta,m)$-expands into $W$ if the following hold.
\begin{enumerate}[(i)]
\item $|N_W(S)|\geq \Delta|S|$ for $S\subseteq V(G)$ with  $|S|< m$.
\item $|N_G(S)\cup S|\geq |S|+ \beta m$ for $S\subseteq V(G)$ with $m\leq|S|\leq |G|/2$.
\end{enumerate}
\end{definition}

The following easy observation shows how we can change the parameters $\Delta$ and $\beta$ while maintaining expansion.
\begin{observation}\label{ObservationExpansionHereditary1}
Suppose that  $G$ $(\Delta, \beta,m)$-expands into $W$.
\begin{enumerate}[(i)]
\item  If $W' \supseteq W$, $\Delta'\leq \Delta$, and $\beta'\leq \beta$, then $G$ $(\Delta', \beta',m)$-expands into $W'$.
\item If $W\subseteq U\subseteq V(G)$  with $|V(G)\setminus U|\leq tm$ then $G[U]$ $(\Delta, \beta-t, m)$-expands into $W$.
\end{enumerate}
\end{observation}

The following lemma shows that graphs whose complement is $K_m^k$-free contain large subgraphs which expand well.
\begin{lemma}\label{LemmaExpanderExistenceMultipartite} 
For all $\beta$, $m$, ${M}$, $\Delta\geq 1$ with $\beta+2< {M}/4$ and $3\Delta<\beta$ the following holds.

Let $G$ be a graph with $\overline{G}$ $K_m^k$-free and $|G|\geq \max(m, {M}(k-1.5)m)$.
Then there exists an integer $k'$, and an induced subgraph $H\subseteq G$ such that the following hold.
\begin{itemize}
\item $\overline H$ is $K_m^{k'}$-free.
\item ${M}(k'-1.5)m-m\leq |H|\leq {M}(k'-1.5)m$. Also we have $|H|\geq m$.
\item $H$ $(\Delta, \beta, m)$-expands into $V(H)$.
\end{itemize}
\end{lemma}
\begin{proof}
The proof is by induction on $k$.  The initial case is when $k = 1$ which holds vacuously since any graph with $m$ vertices contains a copy of $K_m^1$ (by definition $K_m^1$ is just any set of $m$ vertices.)

Assume that for $k\geq 2$ we have a graph $G$ as in the statement of the lemma, and the result holds for all $\hat k<k$. Without loss of generality, we may assume that $|G| = {M}(k-1.5)m$ (by possibly passing to a subgraph of $G$ of this order.)

Suppose that there is a set $S$ with $m \leq |S| \leq |G|/2$ such that $|N_G(S)\cup S| < |S| +  (\beta+1) m$. Let $T = V(G) \setminus (N_G(S)\cup S)$. Using $|S| \leq |G|/2$, $|N_G(S)\cup S| < |S| +  (\beta+1) m$,  $|G| \geq {M}(2-1.5)m$, and  $\beta+2< {M}/4$ we obtain that $|T|\geq m$. We also have that $|S \cup T| = |G|-|N_G(S)\setminus S|\geq {M}(k-1.5)m-(\beta+1) m$. 
Choose $s$ and $t$ maximum integers for which $|S| \geq {M}(s-1.5)m$ and $|T|\geq {M}(t-1.5)m$. 
We certainly have $s,t\geq 1$.
From the maximality of $s$ and $t$, we have $|S\cup T|= |S|+|T| < {M}(s+1-1.5)m + {M}(t+1-1.5)m={M}(s+t-1.5)m+{M} m/2$. 
Combining this with $|S \cup T| \geq {M}(k-1.5)m-(\beta +1)m$, we get  $k-(\beta +1)/{M}\leq s+t+1/2$. Together with $(\beta+1)/{M}+1/2<1$ and the integrality of $k$, $s$, and $t$ this gives $s+t\geq k$. Let $s'\in [1,s]$ and $t'\in [1,t]$ be arbitrary integers with $s'+t'=k$. This ensures $s', t'\leq k-1$.
Since $\overline G$  is $K_m^k$-free and there are no edges between $S$ and $T$, we have that either $\overline{G[S]}$ is $K_m^{s'}$-free or $\overline{G[T]}$ is $K_m^{t'}$-free. 
By induction either $S$ or $T$ contains a subgraph with the required properties.

Now suppose that for every set $S$ with  $m \leq |S| \leq |G|/2$ we have $|N_G(S)\cup S| \geq |S| +  (\beta+1) m$.
Let $S$ be the largest set of vertices in $G$ with $|S|\leq 2m$ for which $|N_G(S) \setminus S| < (\Delta +1)|S|$. 
We have that $|N_G(S)\cup S|\leq (\Delta+2)|S|<|S|+(\beta +1)m$. By our assumption we have that $|S|<m$.

Let $G'=G\setminus S$. We claim that this graph satisfies the conditions of the lemma with $k'=k$.  Certainly $\overline{G'}$ is $K_m^{k'}$-free. Also since $k\geq 2$, we have $|G'|\geq |G|-m= {M}(k'-1.5)m-m\geq m.$
Suppose that we have $S'\subseteq V(G')$ with $|S|\leq |G'|/2\leq |G|/2$. 
If $|S'|\geq m$, then $|N_{G'}(S')\cup S'|\geq |N_{G}(S')\cup S'|-|S|\geq |S'|+  (\beta+1) m-|S|\geq |S'|+  \beta m$.
If $|S'|\leq m$, then by maximality of $S$, we have  $|N_{G}(S'\cup S)\setminus (S'\cup S)|\geq (\Delta+1)|S\cup S'|$ which implies that $|N_{G'}(S')|\geq |N_{G'}(S'\cup S)|-|N_{G'}(S)|\geq|N_{G}(S'\cup S)\setminus (S'\cup S)|-|N_G(S)\setminus S|\geq (\Delta+1)(|S'|+|S|)- (\Delta+1)|S|\geq \Delta|S'|$, proving the lemma.
\end{proof}

Notice that in the above lemma, we can always take $k'\geq 2$, since no graph with $|H|\geq m$ has $\overline H$ is $K_m^{1}$-free.
The following lemma shows that expanders have good connectivity properties.
\begin{lemma}\label{LemmaExpanderPathPrescribedSets}
Suppose that $G$ $(\Delta, \beta, m)$-expands into $W\subseteq V(G)$. Suppose that we have three disjoint sets of vertices $A,B, C\subseteq V(G)$ with $(\Delta-2)|A|\geq |C\cap W|$, $(\Delta-2)|B|\geq |C\cap W|$, and $\beta m \geq 2 |C|$.

Then there is an $A$ to $B$ path $P$ in $G$, avoiding $C$, and with $|P|\leq 8\log m+2|G|/{\beta m}$.
\end{lemma}
\begin{proof}
Let $a=4\log m$ and $b=|G|/{\beta m}$. With this notation, it is sufficient to find an $A$ to $B$ path of length $\leq 2a+2b$.

Set $A^0=A$ and $A^{i+1}=\big(N_G(A^i)\cup A^i\big)\setminus C$ for each $i$.  Using the definition of ``$(\Delta, \beta, m)$-expands'' and $(\Delta-2)|A|\geq |C\cap W|$ we have that $|A^{i+1}|\geq 2|A^i|$ whenever $|A_i|< m$, which implies that $|A^{a}|\geq m$.
Using the definition of ``$(\Delta, \beta, m)$-expands''  and $\beta m \geq 2 |C|$  we have that $|A^{i+1}|\geq  |A^i|+\beta m/2$ whenever $m\leq |A_i|\leq |G|/2$. Combining this with  $|A^a|\geq m$ gives $|A^{a+b}|> b \beta m/2\geq |G|/2$. 

Similarly, letting $B^0=B$ and $B^{i+1}=\big(N_G(B^i)\cup B^i\big)\setminus C$ we have $|B^{a+b}|> |G|/2$. Therefore $A^{a+b}$ and $B^{a+b}$ intersect, giving us the required path.
\end{proof}

\subsection{Embedding trees}\label{SectionTrees}

We'll need a version of a theorem of Friedman and Pippenger \cite{FP} about embedding trees into expanding graphs.
The following lemma is proved in~\cite{BPS}
\begin{lemma}[\cite{BPS}, Lemma 5.2]\label{LemmaEmbedTreesTechnical}
Suppose that we have $\Delta$, $M$, $m$, and $n$ such that $9\Delta m< M$.
Let $X=\{x_1, \dots, x_t\}$ be a set of vertices in a graph $G$ on $n$ vertices. 
Suppose that we have rooted trees $T(x_1),\dots, T(x_t)$ satisfying $\sum_{i=1}^t |T(x_i)|\leq M$ and $\Delta\big(T(x_i)\big) \leq \Delta$ for all $i$.
Suppose that for all $S\subseteq V(G)$ with $m\leq |S|\leq 2m$ we have $|N(S)|\geq M+10\Delta m$, and for $S\subseteq V(G)$ with $|S|\leq m$ we have
\begin{equation}\label{e1}
|N(S)\setminus X|\geq 4\Delta|S\setminus X|+ \sum_{x\in S\cap X} \Big(d_{root}\big(T(x)\big) + \Delta\Big).
\end{equation}
Then we can find disjoint copies of the trees $T(x_1),\dots, T(x_t)$ in $G$ such that for each $i$, $T(x_i)$ is rooted at $x_i$. In addition for all $S\subseteq V(G)$ with $|S|\leq m$, we have
\begin{equation}\label{e2}
|N(S)\setminus \big(T(x_1)\cup\dots\cup T(x_t)\big)|\geq \Delta|S|.
\end{equation}
\end{lemma}

The following version of the above lemma will be easier to apply.
\begin{lemma}\label{LemmaEmbedTrees}
Suppose that we have a graph $G$ and a set $W\subseteq V(G)$ such that  $G$  $(4\Delta, \beta, m)$-expands into $W$ with $20\Delta\leq \beta$. Let $X=\{x_1, \dots, x_t\}=G\setminus W$.

Then for any family of rooted trees $\{T(x_1), \dots, T(x_t) \}$  with $\Delta(T(x_i))\leq \Delta$ and $\sum_{i=1}^{t}|T(x_i)|\leq (\beta -10\Delta) m$ we can find disjoint copies of $T(x_1), \dots, T(x_t)$  in $G$ with $T(x_i)$ rooted at $x_i$ such that $G$    $(\Delta, \beta, m)$-expands into $W\setminus \big(T(x_1) \cup \dots \cup T(x_t)\big)$.
\end{lemma}
\begin{proof}
By setting $M=(\beta -10\Delta)m$, we see that the assumptions of Lemma~\ref{LemmaEmbedTreesTechnical} hold for the family of trees $T(x_1),\dots, T(x_t)$. This allows us to embed the trees $T(x_1),\dots, T(x_t)$  such that $|N_G(S)\setminus (T(x_1)\cup\dots\cup T(x_t))|\geq \Delta|S|$ holds for all  $S\subseteq V(G)$ with $|S|\leq m$.
This shows that part (i) holds of the definition of $G$   $(\Delta,\beta, m)$-expanding into $G\setminus \big(T(x_1) \cup \dots \cup T(x_t)\big)=W\setminus  \big(T(x_1) \cup \dots \cup T(x_t)\big)$.
Part (ii) also holds as a consequence of $G$ $(4\Delta, \beta, m)$-expanding into $W$.
\end{proof}

\subsection{Embedding paths and cycles}\label{SectionPathsAndCycles}
In this section we prove several lemmas about embedding paths and cycles into expanders. They will be the building blocks for the gadgets which we construct in the next section.

The following lemma allows us to connect prescribed vertices together by short paths.
\begin{lemma}\label{LemmaExpanderPathsPrescribedVertices}
Let $G$ be a graph, $\beta, m, t \in \mathbb{N}$, and fix $\ell=4|G|/\beta m+10\log \beta m$.
Suppose that $G$ $(16, \beta, m)$-expands into $W\subseteq V(G)$ with $(\beta -80)m\geq 4\ell t^2+|G\setminus W|$. 
Suppose that we have pairs of vertices $x_1, y_1, x_2, y_2, \dots, x_t, y_t \in G\setminus W$.

Then there are vertex-disjoint paths $P_1, \dots, P_t$ in $G$ with $P_i$ going from $x_i$ to $y_i$ and $|P_i|\leq \ell$.
\end{lemma}
\begin{proof}
Let $X=G\setminus W$ and list the vertices of $X$ as $(x_1, y_1, x_2, y_2, \dots, x_t, y_t, z_1, \dots, z_r)$ for $r=|X|-2t$. We assign a tree $T(v)$ to each $v\in X$ as follows. For $i=1, \dots, t$ the trees $T(x_i)$ and $T(y_i)$ are both rooted binary trees  with $t\ell$ vertices of depth $\leq \lceil\log t\ell\rceil$.  For vertices $z_i$ we let $T(z_i)$ be the tree consisting of a single vertex. 
Notice that $\Delta(T(v))<4$ for all $v$, $G$ $(16, \beta, m)$-expands into $W$, $20\cdot 4\leq \beta$,  and that $\sum_{v\in X} |T(v)|\leq t^2\ell+ |G\setminus W|\leq  (\beta -10\cdot 4)m$.
 Therefore we can apply Lemma~\ref{LemmaEmbedTrees} to $G$ with $\Delta=4$ in order to find disjoint copies of $T(v)$ rooted at all $v\in X$ such that  $G$  $(4, \beta, m)$-expands into $W'=W\setminus \bigcup_{v\in X}T(v)$.

For $i=1, \dots, t$ let $A_i=V(T(x_i))$ and $B_i=V(T(y_i))$. Notice that to prove the lemma it is sufficient to find vertex-disjoint paths $Q_i$ from $A_i$ to $B_i$ internally inside $W'$ of length $\leq \ell-2\lceil\log t\ell\rceil-1$. Indeed once we have such paths, we can join $Q_i$ to the  paths $P_{x_i}$ in $T(x_i)$ and $P_{y_i}$ in $T(y_i)$ from the endpoints of $Q_i$ to $x_i$ and $y_i$ respectively in order to obtain $P_i$ (since $A_i$ and $B_i$ are binary trees of depth $\leq \lceil\log t \ell\rceil$, we know that $e(P_{x_i}), e(P_{y_i})\leq \lceil\log t \ell\rceil$).
We will repeatedly apply Lemma~\ref{LemmaExpanderPathPrescribedSets} to $G$ and $W'$ $t$ times in order to find such paths $Q_1, \dots, Q_t$ of length $\leq \ell-2\lceil\log t\ell\rceil-1$.

Suppose that for some $i\in \{1, \dots, t\}$, we have already found vertex disjoint paths $Q_1, \dots, Q_{i-1}$, each of length $\leq \ell-2\lceil\log t\ell\rceil-1$. Let $C=\left(\bigcup_{j<i}Q_j\right)\cup \left(\bigcup_{j\neq i} A_j\cup B_j\right)$. 
Notice that we have $2|A_i|, 2|B_i|=2t\ell\geq |Q_1|+ \dots+|Q_{i-1}|\geq |C\cap W'|$.
We also have $\beta m \geq 4t^2\ell \geq 2(i-1)\ell + 2t^2\ell \geq 2 |C|$.
Therefore, by Lemma~\ref{LemmaExpanderPathPrescribedSets}, there is a path $Q_i$ from $A_i$ to $B_i$ avoiding $C$ with $|Q_i|\leq  8\log m +2|G|/{\beta m}\leq  \ell-2\lceil\log t \ell\rceil-1$ (the last inequality uses $t \ell \leq \beta m$). 
\end{proof}

The following lemma allows us to find a short cycle $C$ in an expander, such that the graph expands outside $C$.
\begin{lemma}\label{LemmaEmbedShortCycle}
Suppose that we have a nonbipartite graph $G$ which  $(\Delta,\beta,m)$-expands into $W\subseteq G$ with $\Delta\geq 2$.

Then $G$ contains an odd cycle $C$ with $|C|\leq 16\log m+ 4|G|/\beta m$ such that $G$  $(\Delta-5,\beta,m)$-expands into $W\setminus V(C)$.
\end{lemma}
\begin{proof}
Let $C$ be the shortest odd cycle in $G$.
\begin{claim}\label{CycleGeodesicClaim}
For any vertices $x, y \in C$ we have $d_C(x,y)=d_G(x,y)$.
\end{claim}
\begin{proof}
We certainly have $d_C(x,y)\geq d_G(x,y)$.
Suppose for the sake of contradiction that we have $x, y \in C$ with $d_C(x,y)>d_G(x,y)$. Without loss of generality, we may suppose that $d_G(x,y)$ is as small as possible among such pairs of vertices. 
Let $P$ be a $x$ -- $y$ path of length $d_G(x,y)$.

Suppose that $P\cap C$ contains some vertex $z \not \in \{x,y\}$. We have $d_G(x,z)< d_G(x,y)$, so by minimality of $d_G(x,y)$ we have that $d_C(x,z)=d_G(x,z)\leq d_P(x,z)$. Similarly, we obtain $d_C(z,y)=d_G(z,y)\leq d_P(z,y)$. This gives us $d_C(x,y)\leq d_C(x,z)+d_C(z,y)\leq d_P(x,z)+d_P(z,y)= d_G(x,y)$, contradicting $d_C(x,y)>d_G(x,y)$.

Suppose that $P\cap C=\{x,y\}$. Let $Q$ be the $x$ -- $y$ path along $C$ with $|C|$ having the same parity as $|P|$. By replacing $Q$ by $P$ we obtain an odd cycle shorter than $C$ contradicting the minimality of $|C|$.
\end{proof}
From Lemma~\ref{LemmaExpanderPathPrescribedSets} applied with $C=\emptyset$, we have $\mathrm{Diam}(G)\leq 8\log m+2|G|/{\beta m}$ and so Claim~\ref{CycleGeodesicClaim} implies that $|C|\leq 2\mathrm{Diam}(G)\leq 16\log m+ 4|G|/\beta m$.

For any $v \in G$, Claim~\ref{CycleGeodesicClaim} implies that $|N_G(v)\cap C|\leq 5$, since otherwise there would be two vertices $x,y\in N_G(v)\cap C$ with $d_C(x,y)\geq 3>2=d_G(x,y)$. 
Using the fact that $G$  $(\Delta,\beta,m)$-expands into $W$ we obtain that for any $S$ with $|S|< m$ we have $|N_G(S)\cap (W\setminus C)|\geq |N_G(S)\cap W|-|N_G(S)\cap C|\geq (\Delta-5)|S|$. 
This implies that $G$  $(\Delta-5,\beta m, m)$-expands into $W\setminus V(C)$.
\end{proof}

The same proof also proves the following 
\begin{lemma}\label{LemmaEmbedShortPath}
Suppose that we have a  graph $G$ which  $(\Delta,\beta,m)$-expands into $W\subseteq G$ with $\Delta\geq 2$, and we have two vertices $x,y \in V(G)$.

Then there is a path $P$ from $x$ to $y$ with $|P|\leq 16\log m+ 4|G|/\beta m$ such that $G$   $(\Delta-5,\beta,m)$-expands into $W\setminus V(P)$.
\end{lemma}
To prove Lemma~\ref{LemmaEmbedShortPath} one lets $P$ be the shortest $x$ to $y$ path in $G$. The path $P$ ends up having the required properties by the same argument as in of Lemma~\ref{LemmaEmbedShortCycle}.

The following lemma allows us to find a cycle whose length is close to a prescribed value. 
\begin{lemma}\label{LemmaEmbedLongCycle}
Suppose that we have a nonbipartite graph $G$ which  $(\Delta,\beta,m)$-expands into $G$ for $\Delta\geq 20$ and $\beta \geq 8\Delta$. Let $r$ be an odd integer with $r\leq  m$.

Then $G$ contains an odd cycle $C$  with $r+2\leq|C|\leq r+ 16\log m+ 5|G|/\beta m$.
In addition there is an induced subgraph graph $G'$  of $G$ such that $G'$  $(\Delta/4-7,\beta-3, m)$-expands into $V(G)\setminus V(C)$, and $C\setminus V(G')$ is a path of order $r$. 
\end{lemma}
\begin{proof}
By Lemma~\ref{LemmaEmbedShortCycle}, $G$ contains an odd cycle $C_{odd}$ such that $|C_{odd}|\leq 16\log m+ 4|G|/\beta m$ such that $G$  $(\Delta-5,\beta,m)$-expands into $V(G)\setminus V(C_{odd})$. If $|C_{odd}|\geq r+2$, then the lemma holds with $C=C_{odd}$ and $G'$ a subgraph of $G$ formed by deleting $r$ consecutive vertices on $C$ (here $G'$ $(\Delta/4-7,\beta-3, m)$-expands into $V(G)\setminus V(C)$ using $r\leq m$ and Observation~\ref{ObservationExpansionHereditary1}.)
Therefore, suppose that $|C_{odd}|\leq r$, and let $x, y$ be two vertices in $C_{odd}$ at distance $\lfloor|C_{odd}|/2\rfloor$. Notice that this means that there are $x$ to $y$ paths $R^+$ and $R^-$ in $C$ of orders $|C_{odd}|/2+1/2$ and $|C_{odd}|/2-1/2$ respectively.

By Lemma~\ref{LemmaEmbedTrees}, $G$ contains a path $P$  of order $r-|C_{odd}|/2+5/2$ starting with  $x$, with $P\cap C=\{x\}$, such that $G$ $(\Delta/4-2,\beta,m)$-expands into $V(G)\setminus (V(C)\cup V(P))$ (For this application, we have $G=G, W=V(G)\setminus V(C_{odd})$, $X=V(C_{odd})$, $\Delta'=\Delta/4-2$, $\beta=\beta$, and $m=m$. let $T(x)$ be a path  of order $r-|C_{odd}|/2+5/2$, and let $T(x')$ be the single-vertex tree for all $x'\in X\setminus  \{x\}$.) 
 Let $z\neq x$ be the other endpoint of~$P$ and $W=V(G)\setminus (V(C_{odd})\cup V(P))$. 

Suppose that $zy$ is an edge. Joining $R^+$ to $P$ gives a cycle $C$ of order $r+2$ for which the lemma holds with $G'$ a subgraph of $G$ formed by deleting $r$ consecutive vertices on $C$ (here $G'$ $(\Delta/4-7,\beta-3, m)$-expands into $V(G)\setminus V(C)$ using $r\leq m$ and Observation~\ref{ObservationExpansionHereditary1}.)

Suppose that $zy$ is a non-edge. Let $G_1$ be the induced subgraph of $G$ on $(V(G)\setminus(C_{odd}\cup P))\cup \{z,y\}$.
Notice that since  $r\leq m$, Observation~\ref{ObservationExpansionHereditary1} (ii) implies that $G_1$ $(\Delta/4-2,\beta-2,m)$-expands into $W$.
By Lemma~\ref{LemmaEmbedShortPath},   $G_1$   contains an $z$ to $y$ path $Q$ of length $\leq 16\log m+ 4|G|/(\beta -2)m\leq 16\log m+ 5|G|/\beta m$ such that  $G_1$  $(\Delta/4-7,\beta-2,m)$-expands into $W\setminus V(Q)$. Since $zy$ is a nonedge, we have $|Q|\geq 3$.

Notice that $|R^+|$ and $|R^-|$ have different parities.
Therefore we obtain an odd cycle $C$ with $|C|\leq r+ 16\log m+ 5|G|/\beta m$ by joining $Q$ to $P$ to either $R^+$ or $R^-$. 
We have either  $|C|=|R^-|+|P|+|Q|-3$ or $|C| =|R^+|+|P|+|Q|-3$.
Using  $|R^-|=|C_{odd}|/2-1/2$, $|R^+|=|C_{odd}|/2+1/2$, $|Q|\geq 3$, and $|P|=r-|C_{odd}|/2+5/2$ we get $|C|\geq |R^-|+|P|+|Q|-3\geq r+2$ and $|Q|\geq |C|-|R^+|-|P|+3= |C|-r$.  Since $|Q|\geq |C|-r$, we can choose a set  $U$ of $|C|-r$ consecutive vertices on $Q$.
Let $G'$ be the induced subgraph of $G$ on $(V(G_1)\setminus V(Q))\cup U$ to get $G'$ $(\Delta/4-7,\beta-3, m)$-expanding into $W\setminus V(Q)=V(G)\setminus V(C)$ as required (using Observation~\ref{ObservationExpansionHereditary1} (ii).)
\end{proof}

\subsection{Constructing gadgets}\label{SectionConstructingGadgets}
In this section we construct gadgets in graphs whose complement is $K_m^k$-free. The overall goal of this section is to prove Lemma~\ref{LemmaGadgetExistence}.
\begin{figure}
  \centering
     \includegraphics[width=0.7\textwidth]{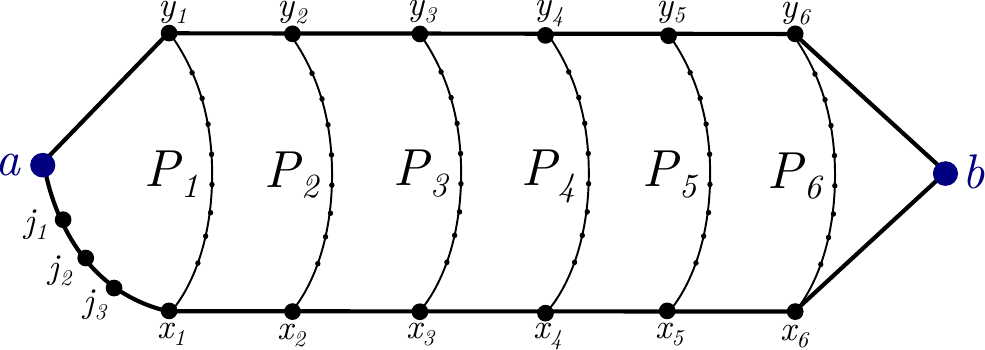}
  \caption{A $3$-gadget \label{FigureGadgetSmall}}
\end{figure}

The following lemma shows that odd $r$-gadgets exist in graphs whose complements are $K_m^k$-free. It also finds two large binary trees attached to the endpoints of the gadget. These binary trees will later be used to join several gadgets together.
\begin{lemma}\label{LemmaEmbedSmallGadget}
Let $m,k,$ and $r$ be integers with $m\geq \max(k^3, 10^9)$, $r$ odd, and $r \leq m$.
Let $G$ be a graph with $\overline{G}$ $K_m^k$-free and $|G|\geq 9100000 km$. 

Then $G$ contains a $r$-gadget $J$ with $|J|\leq r+2000m^{\frac 23}$ with endpoints $a$ and $b$.
In addition there are two disjoint binary trees $T_a$ and $T_b$ in $G$ of order $m$ and depth $\leq \lceil\log  m\rceil$ with $T_a\cap J=\{a\}$ and $T_b\cap J=\{b\}$.
\end{lemma}
\begin{proof}
For this lemma we fix ${M}=9000000$, $\Delta=4000$ and $\beta=1500000$.
See Figure~\ref{FigureGadgetSmall} for an diagram of what kind of $r$-gadget we will find in $G$.

Apply Lemma~\ref{LemmaExpanderExistenceMultipartite} to $G$ in order to find an integer $k'$ and a subgraph $G_1$ of $G$ with $({M}-2)(k'-1.5)m\leq |G_1|\leq {M}(k'-1.5)m$ such that $G_1$ $(\Delta, \beta, m)$-expands into $G_1$ and $\overline{G_1}$ is $K_m^{k'}$-free.
We have $k'\geq 2$, since $|G_1|\geq m$ implies that $\overline{G_1}$ cannot be $K_m^{1}$-free.
Notice that since $|G_1|\leq {M}(k'-1.5)m\leq {M}km\leq {M}m^{\frac 43}$ and ${M}/\beta\leq 6$, we have $|G_1|/\beta m\leq 6m^{\frac 13}$. Notice that $G_1$ is non-bipartite---indeed since $\overline{G_1}$ is $K_m^{k'}$-free,  every set of size $mk'$ in $G_1$ contains an edge which implies that $\alpha(G_1)\leq mk'$ and $\chi(G_1)\geq |G_1|/\alpha(G_1)\geq ({M}-2)(k'-1.5)/k'>10000$.

Apply Lemma~\ref{LemmaEmbedLongCycle} to $G_1$ in order to find an odd cycle $C$ with vertex sequence $a, j_1, \dots, j_r,$ $x_1, x_2, \dots, x_t,$ $b,$ $y_t, y_{t-1}, \dots, y_1$ such that $t\leq 16\log m+ 5|G_1|/\beta m\leq 50m^{\frac 13}$. In addition, we obtain a subgraph  $G_2\subseteq G_1$ which  $(\Delta/5,\beta-3, m)$-expands into $W_2=V(G_1)\setminus C$. Without loss of generality, we may assume that $C$ is labeled so that $\{x_1, x_2, \dots, x_t,$ $b,$ $y_t, y_{y-1}, \dots, y_1,a\}= C\cap G_2$.

Apply Lemma~\ref{LemmaEmbedTrees} to $G_2$ and $W_2$ in order to find two binary trees $T_a$ and $T_b$ internally in $W_2$ of order $m$ and depth $\leq \lceil \log  m\rceil$ with $T_a\cap C=\{a\}$ and $T_b\cap C=\{b\}$ (for this application let $T_x$ be a single-vertex tree for $x\in C\setminus\{a,b\}$.) 
From the application of Lemma~\ref{LemmaEmbedTrees} we have that $G_2$ $(\Delta/20, \beta-3, m)$-expands into $W_2\setminus (T_a\cup T_b)$.
Let $G_3=G_2 \setminus (T_a\cup T_b)$ and $W_3=W_2\setminus (T_a\cup T_b)$. Notice that since $G_2$ $(\Delta/20, \beta-3, m)$-expands into $W_3$ and $|T_a\cup T_b|=2m$, by Obervation~\ref{ObservationExpansionHereditary1} (ii), $G_3$ $(\Delta/20, \beta-5, m)$-expands into $W_3$.

Apply Lemma~\ref{LemmaExpanderPathsPrescribedVertices} to $G_3$,  $W_3$, and the set of pairs $x_1, y_1, \dots, x_t, y_t$ in order to find disjoint paths $P_1, \dots, P_t$  in $G_3$ with $P_i$ joining $x_i$ to $y_i$ and $|P_i|\leq 40m^{\frac 13}$ (for this application we use $\beta' = \beta-5$, $t\leq 50m^{\frac 13}$, $|G_3\setminus W_3|=|G_3\cap C|=2t+2$, $m\geq 10^9$ and $\ell=4|G|/\beta' m+10\log (\beta' m) \leq 40m^{\frac 13}$ which ensure that we have $(\beta' -80)m\geq 4\cdot 50m^{\frac 13}\cdot 50m^{\frac 13} \cdot 40m^{\frac 13}+ 2\cdot50m^{\frac 13}+2\geq 4t^2\ell+ |G_3\setminus W_3|$.)

Let $J= C\cup P_1\cup \dots\cup P_t$. We will show that $J$ is an $r$-gadget satisfying all the conditions of the lemma. 
Notice that the following are both vertex sequences of paths from $a$ to $b$ in $J$:
$$Q_1=a, j_1, j_2, \dots, j_r, x_1, P_1, y_1, y_2, P_2, x_2, x_3, P_3, y_3,\dots, x_t, P_t, y_t, b.$$
$$Q_2=a, y_1, P_1, x_1, x_2, P_2, y_2, y_3, P_3, x_3,\dots, y_t, P_t, x_t, b.$$
We have that $|Q_1|=|J|$ and $|Q_2|=|J|-r$, and so $Q_1$ and $Q_2$ qualify as the two paths in the definition of the $r$-gadget $J$. 
Finally we have $|J|\leq r+t\max_{i=1}^t|P_i|\leq r+2000m^{\frac 23}$.
\end{proof}

The following lemma shows that if the complement of a sufficiently large graph is $K_m^k$-free, then the graph contains a $(\leq t)$-gadget.

\begin{figure}[htb]
  \centering
     \includegraphics[width=1.0\textwidth]{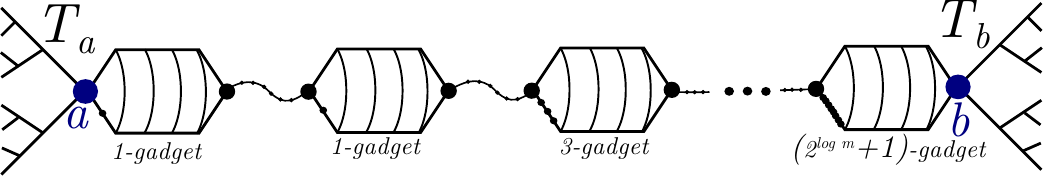}
  \caption{Constructing a $(\leq 2^{\log m})$-gadget in the proof of Lemma~\ref{LemmaEmbedLargeGadget}. \label{FigureGadgetLarge}}
\end{figure}

\begin{lemma}\label{LemmaEmbedLargeGadget}
Let $m,k,$ and $r$ be integers with $m\geq \max(k^3, 10^9)$ and $r \leq \log m$.
Let $G$ be a graph with $\overline{G}$ $K_m^k$-free and $|G|\geq 9500000 km$. 

Then $G$ contains a $(\leq 2^r)$-gadget $J$ with $|J|\leq 2^r+2050\cdot r\cdot m^{\frac 23}$ with endpoints $a$ and $b$.
In addition there are two disjoint binary trees $T_a$ and $T_b$ in $G$ of order $m$ and depth $\leq \lceil\log  m\rceil$ with $T_a\cap J=\{a\}$ and $T_b\cap J=\{b\}$.
\end{lemma}
\begin{proof}
For this lemma we fix ${M}=9500000$, $\Delta=40000$, and $\beta=1500000$.
Apply Lemma~\ref{LemmaExpanderExistenceMultipartite} to $G$  in order to find an integer $k'$ and a subgraph $G'$ of $G$ with $({M}-2)(k'-1.5)m\leq |G'|\leq {M}(k'-1.5)m$ such that $G'$ $(\Delta, \beta, m)$-expands into $G'$ and $\overline{G'}$ is $K_m^{k'}$-free. 
Notice that since $|G'|\leq {M}(k'-1.5)m\leq {M}km\leq {M}m^{\frac 43}$ and ${M}/\beta\leq 7$, we have $|G'|/\beta m\leq 7m^{\frac 13}$.

The strategy of the proof of this lemma is to repeatedly apply Lemma~\ref{LemmaEmbedSmallGadget} in order to find $2^i$-gadgets for $i\in\{1, \dots, r\}$, and join all these gadgets together using Lemma~\ref{LemmaExpanderPathPrescribedSets}. See Figure~\ref{FigureGadgetLarge} for an illustration of what the final $(\leq 2^r)$-gadget  looks like.

\begin{claim}
For $s\leq r$, $G'$ contains a $(\leq 2^s)$-gadget $J$ with $|J|\leq 2^s+ (s+1)2050 m^{\frac 23}$ with endpoints $a$ and $b$.
In addition there are two disjoint binary trees $T_a$ and $T_b$ in $G'$ of order $m$ and depth $\leq \lceil\log  m\rceil$ with $T_a\cap J=\{a\}$ and $T_b\cap J=\{b\}$.
\end{claim}
\begin{proof}
The proof is by induction on $s$.
The initial case ``$s=0$'' follows from Lemma~\ref{LemmaEmbedSmallGadget}. Let $s\geq 1$.
Suppose that we have a $(\leq 2^{s-1})$-gadget $J$ in $G'$ with $|J|\leq 2^{s-1}+s\cdot2050\cdot m^{\frac 23}$ with endpoints $a$ and $b$ as well as two disjoint binary trees $T_a$ and $T_b$ in $G'$ of order $m$ and depth $\leq \lceil\log  m\rceil$ with $T_a\cap J=\{a\}$ and $T_b\cap J=\{b\}$.

The cases  ``$s=1$'' and   ``$s\geq 2$'' are slightly different.
If $s\geq 2$,  apply Lemma~\ref{LemmaEmbedSmallGadget} to $G'\setminus (J\cup T_a\cup T_b)$ in order to find a $(2^{s-1}+1)$-gadget $J'$ in $G'\setminus (J\cup T_a\cup T_b)$  with $|J'|\leq 2^{s-1}+1+2000 m^{\frac 23}$  and with endpoints $a'$ and $b'$ as well as two disjoint binary trees $T'_a$ and $T'_b$ of order $m$ and depth $\leq \lceil\log  m\rceil$ with $T'_a\cap J'=\{a'\}$ and $T'_b\cap J'=\{b'\}$.
If $s=1$, we do the same, except we apply Lemma~\ref{LemmaEmbedSmallGadget} to get a $1$-gadget $J'$ (rather than a $(2^{1-1}+1)$-gadget which we wouldn't be able to obtain from Lemma~\ref{LemmaEmbedSmallGadget} since $2^{1-1}+1$ is even.)

Let $A= T_a$, $B=T'_b$, and $C=J\cup J'\cup T_b\cup T'_a\setminus \{a,b'\}$ to get three sets with $|C|\leq 30000m\leq (\Delta-2)|A|, (\Delta-2)|B|, \beta m /2$. 
Applying Lemma~\ref{LemmaExpanderPathPrescribedSets} to these three sets, gives us a path $P$ from $T_a$ to $T'_b$  avoiding $J\cup J'\cup T_b\cup T'_a\setminus \{a,b'\}$ and satisfying $|P|\leq 8\log m+2|G|/\beta m\leq 22m^{\frac 13}$. Notice that since $T_a$ and $T'_b$ are trees with depth $\leq \lceil\log m\rceil$, there are paths $P_a$ and $P_{b'}$ of length $\leq 2\lceil\log m\rceil$ from $a$ and $b'$ to the two endpoints of $P$. Joining $P$ to $P_a$ and $P_{b'}$ gives a path $Q$ from $a$ to $b'$ of length $\leq 26m^{\frac 13}$.

We claim that $\hat J=J\cup J'\cup Q$ is a  $(\leq 2^s)$-gadget in $G'$ with endpoints $a'$ and $b$. We'll deal with the $s\geq 2$ case first.
Let $t\in\{0, \dots, 2^{s}\}$. We need to find an $a'$ to $b$ path in $\hat J$ of order $|\hat J|-t$. Since  $J$ is $(\leq 2^{s-1})$-gadget, $J$ contains an $a$ to $b$ path $R$ with $|R|=|J|-(t\bmod{2^{s-1}+1})$. Since $J'$ is a $(2^{s-1}+1)$-gadget, $J'$ contains $a'$ to $b'$ paths $R_0$ and $R_1$ with $|R_0|=|J'|$ and $|R_1|=|J'|-2^{s-1}-1$. Now, depending on whether $t\geq 2^{s-1}+1$ or not, either $RQR_0$ or $RQR_1$ is a path of the required length. If $s=1$, then a similar argument works (since both $J$ and $J'$ are $1$-gadgets, we obtain paths $Q_0$ and $Q_1$ in $J$ of orders $|J|$ and $|J|-1$ and paths $R_0$ and $R_1$ in $J'$ of orders $|J'|$ and $|J'|-1$. Now $Q_0QR_0$, $Q_0QR_1$, and $Q_1QR_1$ are paths of lengths $|\hat J|$, $|\hat J|-1$, and $|\hat J|-2$ respectively.)

Notice that as required by the claim, we have the binary trees  $T'_a$ and $T_b$  of order $m$ and depth $\leq \lceil\log  m\rceil$ with $T'_a\cap \hat J=\{a'\}$ and $T_b\cap \hat J=\{b\}$.
Finally, we have $|\hat J|\leq |J|+|J'|+|Q|\leq \left(2^{s-1}+s\cdot 2050\cdot m^{\frac 23}\right)+ \left(2^{s-1}+1+2000m^{\frac 23}\right)+  26m^{\frac 13}\leq 2^s+ (s+1)2050  m^{\frac 23}$ completing the induction step.
\end{proof}
The lemma is immediate from the above claim with $s=r.$
\end{proof}

We are now ready to prove Lemma~\ref{LemmaGadgetExistence}

\begin{proof}[Proof of Lemma~\ref{LemmaGadgetExistence}]
For this lemma we fix ${M}=N_1=10^7$, $\Delta=40000$, $\beta=1500000$, and  $\tilde m= \lambda m$. Notice that $\overline G$ is $K_{\tilde m}^k$-free.
Apply Lemma~\ref{LemmaExpanderExistenceMultipartite} to $G$ with $m=\tilde m$ in order to find an integer $k'$ and a subgraph $G'$ of $G$ with $({M}-2)(k'-1.5)\tilde m\leq |G'|\leq {M}(k'-1.5)\tilde m$ such that $G'$ $(\Delta, \beta, \tilde m)$-expands into $G'$ and $\overline{G'}$ is $K_{\tilde m}^{k'}$-free.  
Notice that since $|G'|\leq {M}(k'-1.5)\tilde m\leq {M}k\tilde m\leq {M}\tilde m^{\frac 43}$ and ${M}/\beta\leq 7$, we have $|G'|/\beta \tilde m\leq 7\tilde m^{\frac 13}$.

Apply Lemma~\ref{LemmaEmbedLargeGadget} twice with $m=\tilde m$ and $r=\lceil\log \tilde m\rceil$ in order to obtain two disjoint $(\leq \tilde m)$-gadgets $J_1$ and $J_2$ in $G'$ with $|J_1|, |J_2|\leq  \tilde m + 2050\cdot \log \tilde m\cdot \tilde m^{\frac 23}$. 
In addition, letting the endpoints of $J_i$ be $a_i$ and $b_i$ we obtain disjoint binary trees $T_{a_i}$ and $T_{b_i}$ of order $\tilde m$ and depth $\leq \lceil\log \tilde m\rceil$ with $T_{a_i}\cap (J_1\cup J_2)=\{a_i\}$ and $T_{b_i}\cap (J_1\cup J_2)=\{b_i\}$. In  order to have disjointness, we first apply Lemma~\ref{LemmaEmbedLargeGadget} to the graph $G'$, and then apply  Lemma~\ref{LemmaEmbedLargeGadget} to the graph $G'\setminus (J_1\cup T_{a_1}\cup T_{b_1})$.

Let $A= T_{a_1}$, $B=T_{a_2}$, and $C=J_1\cup J_2\cup T_{b_1}\cup T_{b_2}\setminus \{a_1, a_2\}$ to get three sets of vertices with $|C|\leq 30000\tilde  m\leq (\Delta-2)|A|, (\Delta-2)|B|, \beta \tilde m/2$. 
Applying Lemma~\ref{LemmaExpanderPathPrescribedSets} to these three sets, gives us a path $P_a$ from $T_{a_1}$ to $T_{a_2}$  avoiding $J_1\cup J_2\cup T_{b_1}\cup T_{b_2}\setminus \{a_1, a_2\}$ and satisfying $|P|\leq 8\log \tilde m+2|G|/\beta\tilde m\leq 22\tilde m^{\frac 13}$. Notice that since $T_{a_1}$ and $T_{a_2}$ are trees with depth $\leq \lceil\log\tilde  m\rceil$, there are paths $P_1$ and $P_2$ of length $\leq 2\log \tilde m$ from $a_1$ and $a_2$ to the two endpoints of $P_a$. Joining $P_a$ to $P_1$ and $P_2$ gives a path $Q_a$ from $a_1$ to $a_2$ of order $\leq 26\tilde m^{\frac 13}$.
By the same argument we can find a disjoint path $Q_b$ from $b_1$ to $b_2$ of order $\leq 26\tilde m^{\frac 13}$ (using $A= T_{b_1}$, $B=T_{b_2}$, and $C=J_1\cup J_2\cup T_{b_1}\cup T_{b_2}\cup Q_a\setminus \{b_1, b_2\}$.)

Now, we have two $(\leq \lambda m)$-gadgets $J_1$ and $J_2$ of order $\leq \lambda m + 2050\cdot \log \lambda m\cdot (\lambda m)^{\frac 23}\leq (\lambda+\mu)m$ (using $\mu m \geq 4100(\lambda m)^{\frac{3}{4}}$), as well as two paths $Q_a$ and $Q_b$ between their endpoints with  $|Q_a|, |Q_b|\leq 26(\lambda m)^{\frac 13}$.

Notice that the following holds 
\begin{equation*}\label{gadgetsandwich}
0\leq |J_1\cup J_2\cup Q_a\cup Q_b|-(\lambda +2\mu)m+2\leq \lambda m. 
\end{equation*}
Indeed, the left hand inequality follows from $|J_1|, |J_2|\geq \lambda m$ and $\lambda\geq 2\mu$, whereas the right hand inequality comes from $\mu m \geq 4100(\lambda m)^{\frac{3}{4}}$ and $|Q_a|, |Q_b|\leq 26(\lambda m)^{\frac 13}, (|J_1|-\lambda m), (|J_2|-\lambda m)< 2050(\lambda m)^{\frac 34}$.

Therefore, since  $J_1$ is a $(\leq \lambda m)$-gadget, there is a path $Q_1$ from $a_1$ to $b_1$ in $J_1$ of order $|J_1|-\big(|J_1\cup J_2\cup Q_a\cup Q_b|-(\lambda +2\mu)m +2\big)$. 
Notice that $|J_2\cup Q_a\cup Q_b\cup Q_1|=(\lambda+2\mu)m-2$ and $|J_2|< (\lambda+\mu)m$. Therefore we can choose two vertices $a$ and $b$ on the path $Q_aQ_1Q_b$ such that the interval $Q$ of $Q_aQ_1Q_b$ from $a$ to $b$ has exactly $\mu m$ vertices. Let $J$ be $J_2$ together with the two segments of $Q_aQ_1Q_b$ outside the internal vertices of  $Q$. Noting that connecting paths to the endpoints of a $(\leq t)$-gadget produces another $(\leq t)$-gadget, we have a $(\leq \lambda m)$-gadget $J$ with $|J|=(\lambda+\mu)m$ and an internally disjoint path $Q$ of order $\mu m$ joining its endpoints.
\end{proof}

\subsection{Gadget cycles}\label{SectionGadgetCycles}
We'll use gadgets by joining many of them into a cycle, and then using the property of a $(\leq k)$-gadget to shorten the cycle into one of prescribed length. The following definition captures the notion of a cycle containing many gadgets on it.

\begin{figure}
  \centering
    \includegraphics[width=0.7\textwidth]{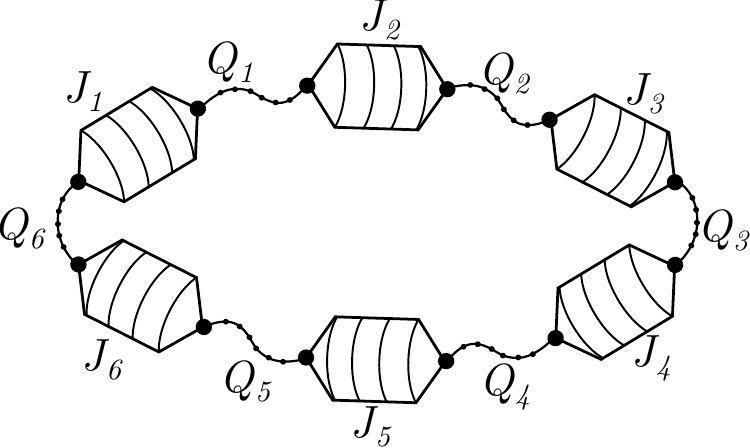}
  \caption{A gadget-cycle. \label{FigureGadgetCycle}}
\end{figure}
\begin{definition}\label{GadgetCycleDefinition}
An $(a,b,m)$-gadget-cycle $C$ is a set of disjoint gadgets $J_1, \dots, J_t$ together with a set of disjoint paths $Q_1, \dots, Q_t$ with the following properties.
\begin{enumerate}[(i)]
\item $J_i$ has endpoints $a_i$ and $b_i$. $Q_i$ goes from $b_i$ to $a_{i+1\pmod t}.$ Other than at these vertices, the paths do not intersect the gadgets.
\item $|J_i|\leq m$ for each $i=1, \dots, t.$
\item $\left|\bigcup_{i=1}^t (J_i\cup Q_i)\right|\geq b$.
\item There is a number $k$ such that each $J_i$ is a $(\leq k)$-gadget with $\left|\bigcup_{i=1}^t (J_i\cup Q_i)\right|-tk\leq a$.
\end{enumerate}
\end{definition}
See Figure~\ref{FigureGadgetCycle} for a diagram of a gadget-cycle. 
Notice that if $C$ is an  $(a,b,m)$-gadget-cycle $C$ then we have $ |C|\geq b$.
Notice that any $(a,b,m)$-gadget-cycle $C$ is also an $(a, |C|, m)$-gadget-cycle.
If $C$ is a gadget-cycle as in Definition~\ref{GadgetCycleDefinition}, we say that it \emph{contains} the gadgets $J_1, \dots, J_t$. If $P_i$ is the path in $J_i$ of order $|J_i|$ for $i = 1, \dots, t$, then we will sometimes identify $C$ with the cycle with vertex sequence $P_1Q_1P_2Q_2\dots P_tQ_t$.

The following simple lemma shows that gadget-cycles contain cycles of all lengths between the parameters $a$ and $b$.
\begin{lemma}\label{LemmaGadgetCycleProperty}
For any $n$ with $a\leq n\leq b$, every $(a,b,m)$-gadget-cycle contains a cycle of length $n$.
\end{lemma}
\begin{proof}
Let $J_1, \dots, J_t$ be $(\leq k)$-gadgets and $Q_1, \dots, Q_t$ paths as in the definition of $(a,b,m)$-gadget-cycle. Choose numbers $k_1, \dots, k_t\in\{1, \dots, k\}$ such that $\left|\bigcup_{i=1}^t (J_i\cup Q_i)\right|-\sum_{k=1}^t k_i=n$ (parts (iii) and (iv) of the definition of ``gadget-cycle" ensure that we can do this). Now since each $J_i$  is a $(\leq k)$-gadget, it contains  a  path $P_i$ between its endpoints of length $|J_i|-k_i$.  Now $\bigcup_{i=1}^t P_i\cup Q_i$ is a cycle of length $n$.
\end{proof}

The following lemma allows us to join two gadget-cycles into a larger gadget-cycle.
\begin{lemma}\label{LemmaJoiningGadgetCycles}
Suppose that we have an $(a_1, b_1, m)$-gadget-cycle $C_1$, an $(a_2, b_2, m)$-gadget-cycle $C_2$, and $r\geq 16$ vertex-disjoint $C_1$ to $C_2$ paths $P_1, \dots, P_r$ of length $\leq \ell$. Then for some $i, j\leq r$ there is an $(a,  b,   m)$-gadget-cycle $C$ with $V(C)\subseteq C_1\cup C_2\cup P_i\cup P_j$,  $|C|\geq (|C_1|+|C_2|)/2$, and
\begin{align*}
a&=a_1+a_2+4m+2\ell,\\
b&=(b_1+b_2)\left(1-\frac{2}{\sqrt{r}}\right).
\end{align*}
\end{lemma}
\begin{proof}
Without loss of generality, we can suppose that $b_1=|C_1|$ and $b_2=|C_2|$. For $i=1, \dots, r$, by possibly replacing each $P_i$ by a shorter path, we can assume that each $P_i$ is internally outside $C_1\cup C_2$.
Let $c^1_1, \dots, c^{|C_1|}_1$ be the vertex sequence of $C_1$ and $c^1_2, \dots, c^{|C_2|}_2$ the vertex sequence of $C_2$.
For a path $P_i$, let $x(P_i)=(s,t)$ where $c^s_1$ and $c^t_2$ are the endpoints of $P_i$. Notice that $x(P_i)\in [1,|C_1|]\times[1, |C_2|]$ for each $i$. There must be two paths $P_i$ and $P_j$ with $x(P_i)$ and $x(P_j)$ within $L^1$ distance $2(|C_1|+|C_2|)/\sqrt{r}$ (otherwise, the $r$  $L^1$-balls of  radius $(|C_1|+|C_2|)/\sqrt{r}$ would all be disjoint. This gives a contradiction to the total volume of these balls being less than $|C_1|\cdot|C_2|$).

Let $S$ be the set of $\leq 2(b_1+b_2)/\sqrt{r}$ vertices of $C_1$ and $C_2$ between the endpoints of $P_i$ and $P_j$.
Let $C$ be the gadget-cycle on $(C_1\cup C_2\cup P_i\cup P_j)\setminus S$ formed by joining $C_1$ and $C_2$ with $P_i$ and $P_j$ and discarding the vertices of $S$. The gadgets of $C$ are all the gadgets of $C_1$ or $C_2$ which are completely contained in $C$. The paths in $C$ are all the other vertices in $C$.

Notice that there are at most four gadgets in $C_1$ and $C_2$ which can intersect $C$ but not be gadgets in $C$ (the only way such a gadget can arise if one of the endpoints of  $P_i$ or $P_j$ is contained in it.)
From this we see  that $C$ is an $(a,b,m)$-gadget-cycle with $a=a_1+a_2+4m+|P_i|+|P_j|\leq a_1+a_2+4m+2\ell$ and $b=|C_1|+|C_2|+|P_i|+|P_j|-|S|\geq b_1+b_2-2(b_1+b_2)/\sqrt{r}$. 
We also have $|C|\geq |C_1|+|C_2|+|P_i|+|P_j|-|S|\geq (b_1+b_2)\left(1-\frac{2}{\sqrt{r}}\right)\geq (b_1+b_2)/2= (|C_1|+|C_2|)/2$.
\end{proof}


\section{Ramsey numbers}\label{SectionRamsey}
In this section we will prove Theorem~\ref{TheoremCnKmkRamsey}. The only results from the previous section which we will use here are Lemmas~\ref{LemmaGadgetExistence},~\ref{LemmaGadgetCycleProperty}, and~\ref{LemmaJoiningGadgetCycles}. We will also employ Theorem~\ref{TheoremRamseyAtLeast} in this section. However, it is worth noting that  the weaker result $R(C_{\leq n}, K_m^k)\leq O(n)$ would also suffice in all our applications of Theorem~\ref{TheoremRamseyAtLeast}.

The structure of this section is as follows. In Section~\ref{SectionExpandersRamsey} we introduce expanders. The expanders which we introduce here are slightly different from the ones we used in the previous section. In Section~\ref{SectionBipartiteRamsey} we prove the special case of Theorem~\ref{TheoremCnKmkRamsey} when $k=2$. Since the full proof of Theorem~\ref{TheoremCnKmkRamsey} is inductive, the ``$k=2$'' case  will serve as the initial case for our induction. In Section~\ref{SectionMultipartiteRamsey} we prove Theorem~\ref{TheoremCnKmkRamsey}.

\subsection{Expanders}\label{SectionExpandersRamsey}
We will use the following notion of expansion.
\begin{definition}
Let $H\subseteq G$ be an induced subgraph of a graph $G$. We say that $H$ is an $(d,m,n)$-expander in $G$ if the following hold.
\begin{enumerate}[(i)]
\item $|N_H(S)|\geq d|S|$ for $S\subseteq V(H)$ with  $|S|< m$.
\item $|N_G(S)\cup S|\geq n$ for $S\subseteq V(H)$ with $|S|\geq m$.
\end{enumerate}
\end{definition}
Notice that if $H$ is an $(d,m,n)$-expander in $G$ and we have $G'\supseteq G$, $d'\leq d$, $m'\geq m$, $n\geq dm'$, and $ n'\leq n$, then  $H$ is an $(d',m',n')$-expander in $G'$.

The following lemma shows that if the complement of a graph is $K_{m,m}$-free, and large sets expand to $n$, then the graph contains a large $(d,m',n)$-expander.
\begin{lemma}\label{LemmaBipartiteExpander}
Suppose that we have integers $n$, $m$, and $d$ with $n> (d+2)(d+3)m$, a graph $G$ and a set of vertices $U\subseteq G$ with $|U|\geq (d+3)^2m$. Suppose that $\overline{G[U]}$ is $K_{m,m}$-free, and that $|N_G(S)\cup S|\geq n$ for every $S\subseteq U$ with $|S|\geq m$. 

Then there is a set $B\subseteq U$ with $|B|<m$ such that $G[U\setminus B]$ is a $(d,(d+2)m,n)$-expander in $G\setminus B$.
\end{lemma}
\begin{proof}
Let $B$ be the largest subset of $U$ with $|B|\leq (d+3)m$ and $|N_U(B)\setminus B|< (d+1)|B|$. Since $|N_U(B)\cup B|\leq (d+2)(d+3)m$ we have that $|U\setminus (N_U(B)\cup B)|\geq m$. Since there are no edges between $B$ and $U\setminus (N_U(B)\cup B)$, the $K_{m,m}$-freeness of $\overline{G[U]}$ implies that $|B|<m$. We show that $G[U\setminus B]$ satisfies (i) and (ii) of the definition of ``$(d,(d+2)m,n)$ expander in $G\setminus B$''.

To see that (i) holds, let $S\subseteq U\setminus B$ be a subset with $|S|< (d+2)m$.
Notice that we have $|N_U(S)\setminus (S\cup B)|\geq (d+1)|S|$ since otherwise $S\cup B$ would be a larger set with $|S\cup B|\leq (d+3)m$ and $|N_U(S\cup B)\setminus (S\cup B)|\leq |N_U(S)\setminus (S\cup B)|+|N_U(B)\setminus B|< (d+1)|S\cup B|$ (contradicting the maximality of $B$). 
This shows that $|N_{U\setminus B}(S)|\geq |N_U(S)\setminus (S\cup B)|\geq d|S|$.

To see that (ii) holds, let $S\subseteq U\setminus B$ be a subset with $|S|\geq (d+2)m$.
We have $|N_U(B)\cap S|\leq |N_U(B)\setminus B|\leq (d+1)|B|\leq (d+1)m\leq |S|-m$, which implies that $|S\setminus N_U(B)|\geq m$.  Therefore, using the assumption of the lemma we get
\begin{align*}
|N_{G\setminus B}(S)\cup S|&\geq|N_{G\setminus B}(S\setminus N_U(B))\cup (S\setminus N_U(B))|\\
                            &= |N_G(S\setminus N_U(B))\cup (S\setminus N_U(B))|\\
                            &\geq n.
\end{align*}
\end{proof}

The following lemma shows that expanders are highly connected.
\begin{lemma}\label{LemmaExpanderConnected}
Let $G$ be a graph with $\overline{G}$ $K_{m,m}$-free and $H$ a $(d+1,m,n)$-expander in  $G$. Then $H$ is $d$-connected.
\end{lemma}
\begin{proof}
Let $x$, $y$ be two vertices in $H$ and $S$ a set of $d-1$ vertices in $H\setminus\{x,y\}$. To prove the lemma, it is sufficient to find an $x$ to $y$ path avoiding $S$.
Define $N^r_{H\setminus S}(v)$ to be the $r$th neighbourhood of a vertex $v\in H$ i.e. the set of all vertices in $H\setminus S$ at distance  $\leq r$ from  $v$ in $H\setminus S$. From the definition of $(d,m,n)$-expander, we have that $|N^r_{H\setminus S}(v)|\geq \min(2^r, m)$ for all $v\in H\setminus S$. 
Therefore  we have $|N^{\log m}_{H\setminus S}(x)|$, $|N^{\log m}_{H\setminus S}(y)|\geq m$. 

We claim that $N^{\log m+1}_{H\setminus S}(x)\cap N^{\log m+1}_{H\setminus S}(y)\neq \emptyset$. If $N^{\log m}_{H\setminus S}(x)\cap N^{\log m}_{H\setminus S}(y)\neq \emptyset$ then this is obvious. Otherwise by $K_{m,m}$-freeness of $\overline{G}$ there is an edge between $N^{\log m}_{H\setminus S}(x)$ and $N^{\log m}_{H\setminus S}(y)$ which is equivalent to $N^{\log m+1}_{H\setminus S}(x)\cap N^{\log m+1}_{H\setminus S}(y)\neq \emptyset$.
We get an  $x$ -- $y$ path avoiding $S$ of length $\leq 2\log m+1$ by joining  paths from $x$ and $y$ to a vertex in $N^{\log m+1}_{H\setminus S}(x)\cap N^{\log m+1}_{H\setminus S}(y)$.
\end{proof}

The same proof as above also gives the following lemma which shows that any two vertices are connected by a short path in an expander.
\begin{lemma}\label{LemmaExpanderShortPath}
Let $G$ be a graph with $\overline{G}$ $K_{m,m}$-free and $H$ a $(3,m,n)$-expander  in  $G$.
Then for any $x,y\in H$, there is an $x$ -- $y$ path $P$ in $H$ with $|P|\leq 3\log m$.
\end{lemma}

It is also possible to connect given vertices by long paths in an expander.
\begin{lemma}\label{LemmaExpanderLongPath}
Let $G$ be a graph with $\overline{G}$ $K_{m,m}$-free and $H$ a $(3,m,n)$-expander  in  $G$ with $|H|\geq 61m$.
Then for any $x,y\in H$, there is an $x$ -- $y$ path $P$ in $H$ with $10m\leq |P|\leq 12m$.
\end{lemma}
\begin{proof}
Notice  that $H-x-y$ contains a cycle $C$ with $|C|\geq 20m$ (eg. by  Theorem~\ref{TheoremRamseyAtLeast}). By Lemma~\ref{LemmaExpanderConnected} combined with Menger's Theorem, there are two disjoint paths $P_x$ and $P_y$ from $x$ and $y$ respectively to $C$. Joining $P_x$ and $P_y$ to the longer segment of $C$ between $P_x\cap C$ and $P_y\cap C$ gives an $x$ to $y$ path $P$ of length $\geq 10m$. 
If $P> 12m$, then by the $K_{m,m}$-freeness of $\overline{G}$, $P$ has a chord whose endpoints are at distance at most $\leq 2m$ on $P$. By repeatedly shortening $P$ with such chords, we obtain a path of length between $10m$ and $12m$.
\end{proof}

\subsection{$R(C_n, K_{m_1,m_2})$}\label{SectionBipartiteRamsey}
The goal of this section is to prove the $k=2$ case of Theorem~\ref{TheoremCnKmkRamsey}. This serves as an initial case of the induction in the full proof of the theorem. 

An important tool which we will need is the P\'osa rotation-extension technique.
Let $P=p_1 p_2 \dots p_t$ be a path in a graph $G$. We say that a path $Q$ is a rotation of $P$ if the vertex sequence of $Q$ is 
$p_1 p_2 \dots p_{i-1} p_t p_{t-1} \dots p_{i+1} p_i$ for some $i$. Notice that for $Q$ to be a path, the edge $p_t p_{i-1}$ must be present. We say that a path $Q$ is \emph{derived} from $P$ if there is a sequence of paths $P_0=P, P_1, \dots, P_s=Q$ with $P_i$ being a rotation of $P_{i-1}$ for each $i$.
We say that a vertex $x$ is an \emph{ending vertex} for $P$ if it is the final vertex of some path derived from $P$. The following lemma from~\cite{BBDK} is a variation of a result of P\'osa from~\cite{Pos}.
\begin{lemma}\label{LemmaEndpointsNeighbourhoodContained}
For $v\in V(G)$,  let $P$ be a maximum length path in $G$ starting at $v$. Let $S$ be the set of ending vertices for $P$. Then $|N_G(S)|\leq 3|S|$.
\end{lemma}

The following lemma could be seen as  a strengthening of the statement that ``$R(C_n, K_{m,m})\leq n-1+m$''---it says that in a graph whose complement is $K_{m,m}$-free which satisfies  certain other conditions, we can connect a given pair of vertices by a path of prescribed length. We will use this  lemma at several points in the proof of Theorem~\ref{TheoremCnKmkRamsey}.
\begin{lemma}\label{LemmaRamseyConnectGivenVertices}
There is a constant $N_2=2\cdot 10^{49}$ such that the following holds.
Let $n$ and $m$ be integers with $n\geq N_2 m$ and $m\geq 8$.
Let $G$ be a graph  with $\overline{G}$ $K_{m,m}$-free and $|N_G(A)\cup A|\geq n$ for every $A\subseteq V(G)$ with $|A|\geq m$. Let $x$  and $y$ be two vertices in $G$ and $P$  an $x$ to $y$ path with $|P|\geq 8m$.

Then there is an $x$ to $y$ path of order $n$ in $G$.
\end{lemma}
\begin{proof}
For this lemma we fix $\mu=10^{20}$, $\lambda= 10^{21}$, $k=2$, and note that  $N_2\geq  2N_1 k\lambda \mu$ where $N_1$ is the constant from Lemma~\ref{LemmaGadgetExistence}.

Without loss of generality, we may assume that $|P|\leq 10m$ (indeed if $|P|>10m$, then by $K_{m,m}$-freeness of $\overline{G}$, $P$ has a chord whose endpoints are at distance $\leq 2m$ along $P$. Shortening $P$ with this chord gives a shorter path of length $\geq 8m$. Therefore there is an $x$ to $y$ path with length between $8m$ and $10m$.)

By Lemma~\ref{LemmaGadgetExistence}, we see that $G\setminus P$ contains a $(\leq \lambda m)$-gadget $J$ of order $(\lambda+\mu)m$ with endpoints $a$ and $b$, together with an internally disjoint path $Q$ of order $\mu m$ from $a$ to $b$.
By $K_{m,m}$-freeness of $\overline{G}$, we can find two disjoint edges from the middle $m+1$ vertices of $Q$ to the middle $m+1$ vertices of $P$. By deleting the segments of $P$ and $Q$ between these edges, we get two paths $P_x$ and $P_y$ of length $\geq 4m$ going from $x$ and $y$ respectively to $a$ and $b$.

Apply Lemma~\ref{LemmaBipartiteExpander} to $G$ with $U=G\setminus (V(P_x)\cup V(P_y)\cup V(J))$ and $d=4$ in order to find a set $B\subseteq U$ with $|B|<m$ such that the subgraph  $H = G\setminus (V(P_x)\cup V(P_y)\cup V(J)\cup B)$   is a $(4, 6m,n)$-expander in $G\setminus B$. Since $|V(P_x)\cup V(P_y)\cup V(J)\cup B|\leq |P|+|Q|+|J|+|B|\leq 10^{22}m$, we have that $|H|\geq m$.

Let $P_x$ have vertex sequence $x=p_0, p_1, p_2, \dots, p_{t}$. By $K_{m,m}$-freeness of $\overline{G}$, there is an edge between some $p_r\in\{p_{m+1}, \dots, p_t\}$ and some vertex $v \in H$. Let $R$ be the longest path in $H$ starting from $v$, and $S$ the set of ending vertices of $P$.  By maximality of $|R|$, we have that $N_H(S)\subseteq R$. Lemma~\ref{LemmaEndpointsNeighbourhoodContained} implies that $|N_H(S)|\leq 3|S|$. By property (i) of $H$ being a  $(4, 6m,n)$-expander in $G\setminus B$, we have that $|S|\geq 6m$. Therefore by property (ii) of $H$ being a $(4, 6m, n)$-expander in $G\setminus B$ we have $|(N_G(S)\cup S)\cap (R\cup P_x\cup P_y\cup J)|= |N_{G\setminus B}(S)\cup S|\geq n$.

Notice that by $K_{m,m}$-freeness of $\overline G$, $S$ has neighbours in $\{p_0, \dots, p_{r-1}\}$. Let $p_i$ be the last neighbour of $S$ in this set. Let $R'$ be a path derived from  $R$ which ends with a neighbour of $p_i$. Let $P'$ be the $x$ to $y$ path formed by joining $p_0,\dots, p_i$ to $R'$  to $p_r, p_{r+1}, \dots, p_t$ to $J$ to $P_y$.
Notice that $(N_G(S)\cup S)\cap (R\cup P_x\cup P_y\cup J)= (N_G(S)\cup S)\cap P'$ (this comes from $P'\setminus (R\cup P_x\cup P_y\cup J)= \{p_{i+1}, \dots, p_{r-1}\}$, and the fact that there are no edges from $S$ to $\{p_{i+1}, \dots, p_{r-1}\}$ by maximality of $i$). 
Together with $|(N_G(S)\cup S)\cap (R\cup P_x\cup P_y\cup J)|\geq n$, this gives  $|P'|\geq n$. The path $P'$ is of the form $P'_x J P'_y$ for some paths $P'_x$ and $P'_y$. Let $P''$ be the shortest path with $|P''|\geq n$ and of the form $P''_xJP''_y$ for some paths $P''_x$ and $P''_y$. Notice that  we must have $|P''|\leq n+5m \leq n+\lambda m$ since otherwise, using $|J|= (\lambda+\mu) m\leq n$ and the $K_{m,m}$-freeness of $\overline G$, either $P''_x$ or $P''_y$ has a chord whose endpoints are at distance $\leq 2m$ on $P''$ (contradicting the minimality of $|P''|$.) Now using the property of the $(\leq \lambda m)$-gadget $J$ we can find an $a$ to $b$ path $J'$ in $J$ of order $|J|-(|P''|-n)$. Joining $J'$ to $P''_x$ and $P''_y$ we obtain an $x$ to $y$ path of order $n$.
\end{proof}

From the above lemma it is easy to find $R(C_n, K_{m_1,m_2})$.
\begin{corollary}\label{CorollaryKmmCnRamsey}
There is a constant $N_2=2\cdot 10^{49}$ such that the following holds.
Let $n$, $m_1, m_2$ be integers with $m_2\geq m_1$, $m_2\geq 8$, and $n\geq N_2 m_2$. Then we have $R(C_n, K_{m_1,m_2})=n+m_1-1$.
\end{corollary}
\begin{proof}
From Lemma~\ref{LemmaRamseyLowerBound}, we have $R(C_n, K_{m_1,m_2})\geq n+m_1-1$. Therefore it remains to show that $R(C_n, K_{m_1,m_2})\leq n+m_1-1$.

Let $G$ be the red colour class of a $2$-edge-coloured $K_{n+m_1-1}$. Suppose that $K_{n+m_1-1}$ contains no blue $K_{m_1,m_2}$ i.e. that $\overline{G}$ is $K_{m_1,m_2}$-free

By Theorem~\ref{TheoremRamseyAtLeast} there are two adjacent vertices $x$ and $y$ with a path of length $\geq n\geq 8m_2$ between them. Since $\overline{G}$ is $K_{m_1,m_2}$-free and $|G|=n+m_1-1$ we have that  $|N_G(A)\cup A|\geq n$ for any $A\subseteq V(G)$ with $|A|\geq m_2$. Also since $m_2\geq m_1$, $\overline G$  is $K_{m_2,m_2}$-free.
Therefore, by Lemma~\ref{LemmaRamseyConnectGivenVertices} applied with $m=m_2$, there is a path of order $n$ from $x$ to $y$ which together with the edge $xy$ gives a cycle of order $n$ in $G$ (and hence a red cycle of order $n$ in the original graph).
\end{proof}

\subsection{$R(C_n, K_{m_1, \dots, m_k})$}\label{SectionMultipartiteRamsey}
Here we prove Theorem~\ref{TheoremCnKmkRamsey}. First we need two intermediate lemmas.

Notice that Theorem~\ref{TheoremCnKmkRamsey} implies that $R(C_n, K_m^k)\leq (k-1)(n-1)+m$.
The following lemma shows that a much better bound holds as long as the red colour class of the $2$-coloured complete graph is highly connected in a certain sense.
\begin{lemma}\label{LemmaConnectedRamsey}
There is a constant $N_3=10^{56}$ such that the following holds.
Suppose that we have $m$, $n$, and $k$ satisfying $n\geq N_3 m$ and $m\geq k^{20}$. Let $G$ be a graph with $|G|\geq 0.07kn+n$.
Suppose that for any two sets of vertices $A$, $B$ of order $2m$, there are at least $k^{20}$ disjoint paths from $A$ to $B$.

Then either $G$ contains a cycle of length $n$ or $\overline G$ contains a copy of $K_m^k$.
\end{lemma}
\begin{proof}
For this lemma we fix $\lambda=10^{24}$, $\mu=10^{21}$, notice that $N_3=10^{49}N_1$ where $N_1$ is the constant from Lemma~\ref{LemmaGadgetExistence}.
For $k=2$, the lemma is weaker than Corollary~\ref{CorollaryKmmCnRamsey}, so we will assume that $k\geq 3$.
Suppose that we have a graph $G$ as in the lemma with $\overline G$ $K_{m}^k$-free. We will find a length $n$ cycle in $G$.

In $G$, select a maximal collection of disjoint $(\leq \lambda  m)$-gadgets of order  $(\lambda +\mu)m$ together with length $\mu m$ paths joining their endpoints i.e. choose disjoint $(\leq \lambda  m)$-gadgets $J_1, \dots, J_t$ of order $(\lambda +\mu)m$ as well as internally disjoint paths  $Q_1, \dots, Q_t$ of order $\mu m$ with $Q_i$ going between the endpoints of $J_i$, such that $t$ is as large as possible. Let $U_1=G\setminus \bigcup_{i=1}^t V(J_i)\cup V(Q_i)$. By maximality of $t$, $G[U_1]$ contains no $(\leq \lambda  m)$-gadget of order $(\lambda+\mu)m$ with a path of length $\mu m$ joining its endpoints. 
By Lemma~\ref{LemmaGadgetExistence}, we have that $|U_1|\leq (N_1 \lambda\mu k)  m$ (since $m\geq k^{20}$, $\lambda=10^{24}$, and $\mu=10^{21}$ imply $m\geq k^3$, $\lambda\geq 2\mu$, and $\mu m\geq 4100(\lambda m)^{\frac34}$.)
Using $N_1\lambda\mu\leq 0.005 N_3$ and $n\geq N_3 m$ we get $|U_1|\leq (N_1 \lambda\mu k)  m\leq 0.005kn$.
This implies  $|J_1\cup \dots\cup J_t\cup Q_1\cup\dots\cup Q_t|\geq 0.06kn+n\geq 1.06n$. 

Construct an auxiliary graph  $H$ on $[t]$ with $ij$ and edge if there are at least $12$ disjoint edges from $Q_i$ to $Q_j$. Using  $|J_1\cup \dots\cup J_t\cup Q_1\cup\dots\cup Q_t|\geq 0.06kn$ and $n\geq N_3 m$, we have $|H|\geq|G\setminus U_1|/(\lambda+2\mu)m\geq 500k$.
The reason for defining this graph $H$ is that paths in $H$ correspond to gadget-cycles in $G$. The following claim makes this precise.
\begin{claim}\label{ClaimPathToGadgetCycle}
Let $P$ be a path in $H$. Then there is an 
$(a, b, 2\lambda m )$-gadget-cycle contained in $\bigcup_{v\in P} (J_v\cup Q_v)$ where
$a=0.01\sum_{v\in P} |J_v\cup Q_v|$ and
$b= 0.99\sum_{v\in P} |J_v\cup Q_v|$.
\end{claim}
\begin{proof}
Without loss of generality, we may assume vertices are labeled so that $P$ has vertex sequence $1,2,\dots,|P|$

Notice that it is sufficient to find a gadget-cycle $C$ containing all the gadgets $J_1, \dots, J_{|P|}$ and with $V(C)\subseteq \bigcup_{i=1}^{|P|}V(J_i)\cup V(Q_i)$. Indeed such a gadget-cycle is always an $(a,   b,   (\lambda +\mu)m)$-gadget-cycle with $a=\sum_{i=1}^{|P|}(|J_i\cup Q_i|-\lambda m)$ and $b=\sum_{i=1}^{|P|}|J_i|$. Using $|J_i|= (\lambda +\mu)m$ and $|Q_i|= \mu m$, we have that $a\leq 0.01\sum_{i=1}^{|P|}|J_i\cup Q_i|$ and $b\geq 0.99\sum_{i=1}^{|P|}|J_i\cup Q_i|$ (for these we use $\lambda\geq 200\mu$.) It remains to show that such a gadget-cycle containing all the gadgets $J_1, \dots, J_{|P|}$ exists.

For each $i$, let $M_i$ be the matching of size $12$ from $Q_i$ to $Q_{i+1}$ (which exists since $\{i,i+1\}$ is an edge in $H$). Fix some orientation of $Q_i$ for each $i$.

Notice that for any two sets of distinct numbers $S$ and $T$, there are two subsets $S'\subseteq S$ and $T'\subseteq T$ with $|S'|\geq|S|/2-1$ and $|T'|\geq|T|/2-1$ for which we  have  either ``$s< t$ for all $s\in S', t\in T'$'' or ``$t<s$ for all $s\in S', t\in T'$''. 
For $i=1,2,\dots, |P|-1$ we apply this repeatedly with $S=Q_i\cap M_{i-1}$ and $T=Q_i\cap M_{i}$ in order to obtain new matchings $M'_1\subset M_1, \dots, M'_j\subset M_j $ of size $2$ with the property that the endpoints of $M'_{i-1}$ in $Q_i$ are either all to the left or all to the right of  the endpoints of $M'_i$ in $Q_i$.

Now, for each $i$, we delete the segment of $Q_i$ between the endpoints of $M'_{i-1}$ and the segment of $Q_i$ between the endpoints of  $M'_i$. Adding the edges of $M_{1}, \dots, M_{|P|}$ to the graph produces the required gadget-cycle containing all the gadgets $J_1, \dots, J_{|P|}$. 
\end{proof}
We will often use the fact that the gadget-cycle produced by Claim~\ref{ClaimPathToGadgetCycle} has order at least $0.99\sum_{v\in P} |J_v\cup Q_v|$ (which holds since any $(a,b,m)$-gadget-cycle has order at least $b$).

Using the $K_{m}^k$-freeness of $\overline{G}$ we obtain that $H$ has small independence number.
\begin{claim}\label{ClaimIndependenceOfH}
$\alpha(H)\leq k-1$.
\end{claim}
\begin{proof}
Suppose for the sake of contradiction that $H$ contains an independent set $I$ of order $k$. 
Let for $i,j\in I$, let $M_{i,j}$ be a maximal matching in $G$ between $Q_i$ and $Q_j$. From the definition of edges in $H$ we have that $M_{i,j}\leq 12$ for $i,j\in I$.
For $i \in I$, let $Q'_i=Q_i\setminus \bigcup_{i,j\in I} V(M_{i,j})$. Using $m\geq k^{10}$ we have $|Q_i|=\mu m\geq m+12k$, which implies that $|Q'_i|\geq m$.
By maximality of $M_{i,j}$ there are no edges between $Q'_i$ and $Q'_j$. But this means that $\overline{G\left[\bigcup_{i\in I} Q'_i\right]}$ contains a copy of $K_m^k$ contradicting the $K_m^k$-freeness of $\overline G$.
\end{proof}

The following is a variant of the well known fact that a graph can be covered by $\alpha(G)$ vertex-disjoint paths.
\begin{claim}\label{ClaimDisjointPaths}
There are $k-1$ vertex-disjoint paths $P_1, \dots, P_{k-1}$ in $H$ with $|H|-|P_1|-\dots-|P_{k-1}|\leq 200k$ and $|P_i|\geq 200$ for $i=1, \dots, k-1$.
\end{claim}
\begin{proof}
Choose vertex disjoint paths $Q_1, \dots, Q_t$ in $H$ covering $V(H)$ with $t$ as small as possible. Without loss of generality, suppose that we have $|Q_1|\leq |Q_2|\leq \dots \leq |Q_t|$.
By minimality of $t$, we have that the starting vertices of $Q_1, \dots, Q_t$ form an independent set (otherwise we could join two of the paths together to obtain a smaller collection of paths.) Claim~\ref{ClaimIndependenceOfH} implies that $t\leq k-1$. 

Let $r$ be the index with $|Q_{r-1}|< 200$ and $|Q_r|\geq 200$ (possibly with $r=0$.) Let $U_1=Q_1\cup\dots\cup Q_{r-1}$ to obtain a set with $|U_1|\leq 200k$.
Using $|H|\geq 500k$, it is possible to break some of the paths $Q_r, \dots, Q_{k-1}$ into shorter paths in order to obtain a collection of exactly $k-1$ paths $P_1, \dots, P_{k-1}$ of orders $\geq 200$ (to do this notice that in any collection of $<k-1$ paths of total order $\geq 500k$, there must be a path of order $\geq 500$.)
\end{proof}

Let $P_1, \dots, P_{k-1}$ be the paths from the above claim and assume that they are ordered such that $|P_1|\geq |P_2|\geq \dots \geq |P_{k-1}|$. 
Let $U_2=\bigcup_{v\in H\setminus(P_1\cup \dots\cup P_{k-1})} J_v\cup Q_v$, and observe that from Claim~\ref{ClaimDisjointPaths} we have that $|U_2|\leq 200(\lambda+2\mu)mk\leq 0.005kn$.

Suppose that $\sum_{v\in P_1} |J_v\cup Q_v|\geq 2n$. Then since for each $v$, $|J_v\cup Q_v|\leq (\lambda +2\mu)m\leq n$ there is a path $P\subseteq P_1$ with $3n\geq \sum_{v\in P} |J_v\cup Q_v|\geq 2n$. 
By Claim~\ref{ClaimPathToGadgetCycle}, there is a $\big(0.01\sum_{v\in P} |J_v\cup Q_v|,  0.99\sum_{v\in P} |J_v\cup Q_v|, 2\lambda m \big)$-gadget-cycle in $G$. Notice that we have
$$0.01\sum_{v\in P} |J_v\cup Q_v|\leq 0.01\cdot 3n\leq n\leq 0.99\cdot 2n\leq 0.99\sum_{v\in P} |J_v\cup Q_v|.$$ 
Lemma~\ref{LemmaGadgetCycleProperty} implies that $G$ contains a cycle of length $n$.

Suppose that  $\sum_{v\in P_1} |J_v\cup Q_v|\leq 2n$. For $i=1, \dots, k-1$, let  $C^g_i$ be the gadget-cycle produced out of the path $P_i$ using Claim~\ref{ClaimPathToGadgetCycle}. We have
\begin{equation}\label{EqGgiGadgetCycle}
\text{$C^g_i$ is a $\left(0.01\sum_{v\in P_i} |J_v\cup Q_v|,  0.99\sum_{v\in P_i} |J_v\cup Q_v|, 2\lambda m \right)$-gadget-cycle.}
\end{equation}
Let $U_3=\bigcup_{i=1}^{k-1} \left(\bigcup_{v\in P_i}(J_v\cup Q_v)\right)\setminus C^g_i$. Notice that from~(\ref{EqGgiGadgetCycle}) we have   $|C^g_i|\geq 0.99\left|\bigcup_{v\in P_i}(J_v\cup Q_v)\right|$ for $i=1, \dots, k-1$, which together with $\sum_{v\in P_i} |J_v\cup Q_v|\leq 2n$ implies that $|U_3|\leq  0.02kn$.
Let $U=U_1\cup U_2 \cup U_3=G\setminus \bigcup_{i=1}^{k-1}C_i^g$ to get a set with $|U|\leq 0.03kn$. 
Notice that as a consequence of (\ref{EqGgiGadgetCycle}), $|U|\leq 0.03kn\leq 0.1|G|$, and $|P_1|\geq \dots \geq |P_{k-1}|$ we have $|C_1^g|\geq |G|/2k$.
Using $|P_i|\geq 200$ and (\ref{EqGgiGadgetCycle}) we have that for all $i$
\begin{equation}\label{EqCgiLarge}
|C^g_i|\geq 0.99\sum_{v\in P_i} |J_v\cup Q_v|\geq 0.99\cdot 200(\lambda m+2\mu m-2)\geq 2m
\end{equation}

For a permutation $\sigma$ of $[k-1]$, we set $S^{\sigma}_i=\sum_{j=1}^i\sum_{v\in P_{\sigma(j)}}|J_{v}\cup Q_{v}|$ for $i =1, \dots, k-1$. 
Notice that $S^{\sigma}_{k-1}=|G|-|U_1\cup U_2|$ always holds. 
Using the fact that $|P_i|\geq 200$ for each $i$, we always have $S^{\sigma}_i\geq 199(\lambda+2\mu)im$.
\begin{claim}\label{ClaimGadgetCycleRecursion}
There is a sequence of gadget-cycles $D_1, \dots, D_{k-1}$  as well as a permutation $\sigma$ of $[k-1]$ with the following properties:
\begin{enumerate}[(a)]
\item $\sigma(1)=1$.
\item For each $i$ we have  $D_i\subseteq U\cup C^g_{\sigma(1)}\cup C^g_{\sigma(2)}\cup\dots\cup C^g_{\sigma(i)}$.
\item For each $i$ we have $|D_i|\geq 2m$.
\item $D_i$ is an $(a_i,b_i,2\lambda m)$-gadget cycle for
\begin{align*}
a_i&=0.01S^{\sigma}_i+8(i-1)\lambda m +2(i-1)|G|/k^{7}\\
b_i&=0.99(1-2k^{-6})^{i-1}S^{\sigma}_i.
\end{align*}
\end{enumerate}
\end{claim}
\begin{proof}
Set $D_1=C_1^g$ and $\sigma(1)=1$. Now for $i=1$,  (a) and (b) hold trivially, (c) comes from (\ref{EqCgiLarge}), and (d) is equivalent to the ``$i=1$'' case of ~(\ref{EqGgiGadgetCycle}). 
For $i\geq 2$ we will recursively construct $D_{i}$, $\sigma(i)$ from $D_1, \dots, D_{i-1}$, and $\sigma(1), \dots, \sigma(i-1)$. Suppose that we have already constructed $D_1, \dots, D_{i-1}$, and $\sigma(1), \dots, \sigma(i-1)$ satisfying (a) -- (d). We construct $D_{i}$ and $\sigma(i)$ as follows:

By (c), we have $|D_{i-1}| \geq 2m$ and by (\ref{EqCgiLarge}) we have $|\bigcup_{j\in[k-1]\setminus\{\sigma(1), \dots, \sigma(i-1)\}}C^g_i|\geq 2m$. Using the assumption of the lemma,  we   find at least $k^{20}$ disjoint paths from $D_{i-1}$ to $\bigcup_{j\in[k-1]\setminus\{\sigma(1), \dots, \sigma(i-1)\}}C^g_i$ internally contained outside these sets. Since the paths are all disjoint, there is a subcollection of $k^{20-7}$ of them with length $\leq |G|/k^{7}$. In addition, there is a further subcollection of $k^{20-7-1}$ of them which go from $D_{i-1}$ to $C^g_j$ for some particular $j$. 
To get $D_i$, we apply Lemma~\ref{LemmaJoiningGadgetCycles} to this collection of $k^{20-7-1}$ paths, the gadget-cycles $C_1=C^g_j$ and $C_2=D_{i-1}$ and with the parameters $m'=2\lambda m$,  $r=k^{20-7-1}$ and $\ell = |G|/k^{7}$. We set $\sigma(i)=j$.

Now (b) holds as a consequence of ``$V(C)\subseteq C_1\cup C_2\cup P_i\cup P_j$'' in Lemma~\ref{LemmaJoiningGadgetCycles} and (c) holds as a consequence of ``$|C|\geq (|C_1|+|C_2|)/2$'' in Lemma~\ref{LemmaJoiningGadgetCycles}.

Recall that $C_1=C^g_j$ is a $(0.01(S^{\sigma}_{i}-S^{\sigma}_{i-1}), 0.99(S^{\sigma}_{i}-S^{\sigma}_{i-1}), 2\lambda m)$-gadget-cycle  by (\ref{EqGgiGadgetCycle})  and $C_2=D_{i-1}$ is a $(a_{i-1}, b_{i-1}, 2\lambda m)$-gadget-cycle by (d) holding for $D_{i-1}$. Thus (d) holds for $D_i$ from the application of Lemma~\ref{LemmaJoiningGadgetCycles} together with $m'=2\lambda m$, $r=k^{12}$,  $\ell=|G|/k^7$, and ``$(b_{i-1}+0.99(S^{\sigma}_{i}-S^{\sigma}_{i-1}))(1-2k^{-6})\geq b_{i}$''.
\end{proof}

From here, fix $\sigma$ to be the permutation from Claim~\ref{ClaimGadgetCycleRecursion}.
Notice that $S^{\sigma}_{1}\geq |C_1^g|\geq |G|/2k$  implies $2(i-1)|G|/k^{7}\leq 0.1S^{\sigma}_1$, while $S^{\sigma}_i-S^{\sigma}_1\geq 199(\lambda+2\mu)(i-1)m$ implies $8(i-1)\lambda m\leq 0.1 (S^{\sigma}_i-S^{\sigma}_1)$.
Combining these gives  $8(i-1)\lambda m +2(i-1)|G|/k^{7}\leq  0.1 S^{\sigma}_i$ and hence $a_i\leq 0.11S^{\sigma}_i$.
We also have $b_i\geq 0.99(1-2k^{-6})^{k}S^{\sigma}_i\geq 0.99(1-2k^{-6+1})S^{\sigma}_i\geq 0.91S^{\sigma}_i$ (using $k\geq 3$.) Putting these together we have that $a_i\leq 0.25 b_i$ for all $i$. 

Since  $S^{\sigma}_i-S^{\sigma}_{i-1}=\sum_{v\in P_{\sigma(i)}} |J_v\cup Q_v|\leq 2n$ for all $i$, we have that $b_i\leq 0.99(1-2k^{-6})^{i-2}S^{\sigma}_i = b_{i-1}+ 0.99(1-2k^{-6})^{i-2}(S^{\sigma}_i- S^{\sigma}_i)\leq b_{i-1}+2n$. 
Also, using $k\geq 3$ we have $S^{\sigma}_{k-1}=\sum_{s=1}^{k-1} |J_s\cup Q_s|=|G|-|U_1\cup U_2|\geq (1+0.07k)n-0.01kn\geq 1.16n$ which implies $b_{k-1}\geq 0.91\cdot 1.15n\geq n$.
Combining these we get that there is some $i$ for which $n\leq b_i\leq 3n$ and hence $a_i\leq 0.25\cdot 3n\leq n$. By Lemma~\ref{LemmaGadgetCycleProperty}, $D_i$ contains a cycle of length $n$.
\end{proof}

The following lemma could be seen as a structural statement of the form ``If $N$ is close to $R(C_n, K_m^k)$ and $K_N$ is $2$-colored without red cycles $C_n$ and blue $K_{m}^k$ then the colouring on $K_N$ must be close to the extremal colouring''.
\begin{lemma}\label{LemmaPartition}
There is a constant $N_3=10^{58}$ such that the following holds.
Suppose that $n\geq N_3 m$, $m\geq k^{21}$, $k\geq 2$ and $G$ is a graph with $|G|\geq  (k-1)n$,  $G$ $C_n$-free, and $\overline G$ $K_m^k$-free.
Then $V(G)$ can be partitioned into  sets $A_1, \dots, A_{k-1}$, and $S$ such that the following hold.
\begin{enumerate}[(i)]
\item $|A_i|\geq m$ for $i=1, \dots, k-1$.
\item There are no edges between $A_i$ and $A_j$ for $i \neq j$.
\item $\overline{G[A_i]}$ is $K_{m,m}$-free for $i=1, \dots, k-1$.
\item $|S|\leq k^{11}$.
\end{enumerate}
\end{lemma}

\begin{proof}
 We construct a sequence of graphs $G_0, G_1, \dots$ recursively as follows.
Let $G_0=G$. If $G_i$ contains three sets $A$, $B$, $S_i$ with $|A|, |B|\geq m$, $|S_i|\leq k^{20}$, such that $A$, $B$, and $S_i$ all lie in the same connected component of $G_{i}$ and $S_i$ separates $A$ from $B$, then let $G_{i+1}=G_i\setminus S_i$. 
Notice that since $\overline{G_0}$ is $K_{m}^k$-free, we must have $G_{k}=G_{k-1}$. 

Choose a partition of $V(G_k)$ into sets $A_1, \dots, A_{t}$ such that for each $i$, we have $|A_i|\geq m$, there are no edges between $A_i$ and $A_j$ for $i\neq j$, and $t$ is as large as possible.
Notice that any $A_i$ with $|A_i|\geq 3m$, must have a connected component $C_i$ of order at least $|A_i|-m+1$ (otherwise $A_i$ can be split into sets of order $\geq m$ with no edges between them, contradicting the maximality of $t$).
From this we obtain that, any $A_i$ with $|A_{i}|\geq 5m$ must have the property that ``for any two subsets $A, B\subseteq A_i$ of order $\geq 2m$, there are at least $k^{20}$ disjoint paths from $A$ to $B$ in $A_i$". Indeed otherwise, by Menger's Theorem there would be a set $S'$ of size $\leq k^{20}$ separating $A\cap C_i$ from $B\cap C_i$, contradicting $G_k=G_{k-1}$. 
 Let $S=S_1\cup\dots\cup S_{k-1}=V(G)\setminus V(G_{k})$ to get a set with $|S|\leq k^{21}$.

For $i=1, \dots, t$, let $x_i=\min(0,|A_i|-n)$.
\begin{claim}\label{ClaimPartition}
$\overline{G_k}$ contains $K_m^r$ for $r=t+\left\lfloor\frac{4}{n}\sum_{i=1}^t x_i\right\rfloor$.
\end{claim}
\begin{proof}
Without loss of generality, suppose that $A_1, \dots, A_t$ are ordered so that $x_1, \dots, x_a\leq m$ and $x_{a+1}, \dots, x_t\geq m$  for some integer $a$.

Using Lemma~\ref{LemmaConnectedRamsey},  we see that when $x_i\geq 0.25n$, $\overline{G[A_i]}$ contains a $K_m^j$ for $j = 1+\lceil 4 x_i/n\rceil$ (first notice that $\left\lfloor x_i/0.07n \right\rfloor\geq 1+\lceil 4 x_i/n\rceil$ for $x_i\geq 0.25n$. This implies that $|A_i|=x_i+n\geq 0.07jn+n$, and so the assumptions of Lemma~\ref{LemmaConnectedRamsey} hold for $G[A_i]$ with $k'=j$).
By Corollary~\ref{CorollaryKmmCnRamsey} we know that $\overline{G[A_i]}$ contains a $K_m^2$ whenever $|A_i|\geq n+m-1$. This implies that when $m\leq x_i< 0.25n$ then $\overline{G[A_i]}$ contains a copy of  $K_m^{1+\lceil 4 x_i/n\rceil}=K_m^j=K^2_m$.
Putting the above observations together, we obtain that  $\overline{G[A_{a+1}\cup\dots\cup A_t]}$ contains a $K_m^{t-a+\sum_{i={a+1}}^t\lceil4 x_i/n\rceil}$.

By Corollary~\ref{CorollaryKmmCnRamsey} and the fact that $|A_i|\geq m$, we know that $\overline{G[A_i]}$ contains a $K_{m, x_i}$ whenever $x_i\leq m$.
Since there are no edges between any of these $K_{m, x_1}, \dots, K_{m, x_a}$, their  union consists of a $K_m^a$ together with a $K_{x_1, \dots, x_a}$. Using $x_1, \dots, x_a\leq m$, we have that $K_{x_1, \dots, x_a}$ contains a $K_{m}^{\lfloor \sum_{i=1}^a x_i/2m\rfloor}$. Together with $K_m^a$ this gives a copy of $K_m^{a+\lfloor \sum_{i=1}^a x_i/2m\rfloor}$ in $\overline{G[A_1\cup\dots\cup A_a]}$. Since $4/n\leq 1/2m$, this contains a copy of $K_m^{a+\lfloor \sum_{i=1}^a 4x_i/n\rfloor}$.

Now, we have found a $K_m^{t-a+\sum_{i={a+1}}^t\lceil 4 x_i/n\rceil}$ and a disjoint $K_m^{a+\lfloor \sum_{i=1}^a 4x_i/n\rfloor}$. Putting these two together, and using $\lfloor x\rfloor + \lceil y\rceil\geq\lfloor x+y\rfloor$  we obtain a $K_m^{t+\left\lfloor\frac{4}{n}\sum_{i=1}^t x_i\right\rfloor}$ as required.
\end{proof}

From Claim~\ref{ClaimPartition} and the $K_m^k$-freeness of $\overline {G}$, we obtain that $t+\left\lfloor\frac{4}{n}\sum_{i=1}^t x_i\right\rfloor\leq k-1$.
We also have that $tn+\sum_{i=1}^t x_i\geq |G|-|S|\geq (k-1)n-k^{21}$.
Putting these together, we get  $\frac{1}{n}\sum_{i=1}^t x_i\geq \left\lfloor\frac{4}{n}\sum_{i=1}^t x_i\right\rfloor-\frac{k^{21}}{n}$. Together with $n>10k^{21}$, this gives $\sum_{i=1}^t{x_i}<n/2$. Combined with  $tn+\sum_{i=1}^t x_i\geq (k-1)n-k^{11}$ this implies $t-k+1\geq -\frac{1}{2}-\frac{k^{11}}{n}$. Since $t-k+1$ is an integer, this implies that $t\geq k-1$.  From $t+\left\lfloor\frac{4}{n}\sum_{i=1}^t x_i\right\rfloor\leq k-1$ we obtain that $t=k-1$.
 The $K_m^k$-freeness of $\overline {G}$ implies that each $\overline {G[A_i]}$ is $K_{m,m}$-free, proving the lemma.
\end{proof}

We can now prove the main result of this paper.
\begin{proof}[Proof of Theorem~\ref{TheoremCnKmkRamsey}]
From Lemma~\ref{LemmaRamseyLowerBound} we have that $R(C_n, K_{m_1, \dots, m_k})\geq (n-1)(k-1)+m_1$. Therefore, it remains to prove the upper bound. Fix  $N_3=10^{60}$. Let $n,\hat k,  m_1, \dots, m_{\hat k}$ be numbers with $n\geq N_3 m_{\hat k}$, $m_{\hat k}\geq m_{\hat{k}-1}\geq \dots \geq m_1$ and $m_i\geq i^{22}$ for $i=1, \dots, \hat k$.

We prove that $R(C_n, K_{m_1, \dots, m_k})\leq (n-1)(k-1)+m_1$  for $k=2, \dots, \hat k$ by induction on $k$. The initial case is when $k = 2$ which comes from Corollary~\ref{CorollaryKmmCnRamsey}.
Therefore assume that $k \geq 3$ and that we have $R(C_n, K_{m_1, \dots, m_{k-1}}) \leq  (n-1)(k-2)+m_{1}.$ Let $K$ be a $2$-edge-coloured complete graph on  $(n-1)(k-1)+m_1$ vertices.
Suppose, for the sake of contradiction that $K$ contains neither a red $C_n$ nor a blue $K_{m_1, \dots, m_k}$. 
Let $G$ be the subgraph consisting of the red edges of $K$. 

\begin{claim}\label{ClaimExpandToN}
$|N_G(W)\cup W| \geq n$ for every $W \subseteq G$ with $|W| \geq m_k$.
\end{claim}
\begin{proof}
Suppose that $|N_G(W)\cup W| \leq n-1$ for some $W$ with $|W|\geq m_k$.
Let $K'=K\setminus (N_G(W)\cup W)$ to get a graph with $|K'|\geq (n-1)(k-2)+m_1$. By induction $K'$ contains either a red $C_n$ or a blue $K_{m_1, \dots, m_{k-1}}$. In the former case, we have a red $C_n$ in $K$, whereas in the latter case we have a blue $K_{m_1, \dots, m_k}$ formed from the copy of $K_{m_1, \dots, m_{k-1}}$ together with $W$.
\end{proof}

 Set $m=m_k$, and notice that $\overline G$ contains no blue $K_m^k$.
Apply Lemma~\ref{LemmaPartition} to $G$ in order to partition it into sets $A_1, \dots, A_{k-1}$ and $S$ satisfying (i) -- (iv). Notice that from condition (ii) of Lemma~\ref{LemmaPartition} and Claim~\ref{ClaimExpandToN}, we have $|(N_G(W)\cup W)\cap (A_i\cup S)|\geq n$ for any $W\subseteq A_i$ with $|W|\geq m$. Combined with $|S|\leq k^{21}\leq m$ and $n\geq N_3 m$, this implies that $|A_i|\geq 10^{53}m$ for each $i$.
For $i=1, \dots, k-1$, apply Lemma~\ref{LemmaBipartiteExpander} with $U=A_i$, $G=G[A_i\cup S]$, and $d=3$ in order to find  subsets $H_i\subseteq A_i$ with $|H_i|\geq |A_i|-m$ such that $G[H_i]$ is a $(3,5m,n)$-expander in $G[H_i\cup S]$. Let $G'=G[H_1\cup \dots\cup H_{k-1}\cup S]$.

Suppose that for some $i$ and $j$, there exist two vertex-disjoint paths  from $H_i$ to $H_j$ in $G'$. Let $P_1$ and $P_2$ be two such paths with $|P_1|+|P_2|$ as small as possible. Using  Lemma~\ref{LemmaExpanderShortPath} we have that $|P_1\cap H_s|, |P_2\cap H_s|\leq 3\log 5m$ for all $s$ (since if we had $|P_1\cap H_s|>3\log 5m$ then Lemma~\ref{LemmaExpanderShortPath} would give a shorter path in $H_s$ between the first and last vertex of $P_1$ in $P_1\cap H_s$.)  Together with $m\geq k^{22}$, this implies  $|P_1|, |P_2|< 3k\log 5m \leq m$. Let $p_1^i$ and $p_2^i$ be the endpoints of $P_1$ and $P_2$ in $H_i$ and let $p_1^j$ and $p_2^j$ be the endpoints of $P_1$ and $P_2$ in $H_j$. 
By Lemma~\ref{LemmaExpanderLongPath} applied with $m'=5m$, there is an $p_1^i$ to $p_2^i$ path $Q_i$  in $H_i$ as well as a $p_1^j$ to $p_2^j$ path $Q_j$ in $H_j$ with $50m\leq |Q_i|, |Q_j|\leq 60m$.
Notice that we have $n-62m\leq n-|Q_i\cup P_1\cup P_2|+2\leq  n-50m$ and ``$|N_{H_j}(W)\cup W|\geq n-|S|\geq n-m$ for $W\subseteq H_j$ with $|W|\geq 5m$.'' 
Therefore, we can apply  Lemma~\ref{LemmaRamseyConnectGivenVertices} to $H_j$ with $P=Q_j$, $m'=5m$, and $n'=n-|Q_i\cup P_1\cup P_2|+2$ in order find a $p_1^j$ to $p_2^j$ path $Q_j'$ in $H_j$ with $|Q_j'|=n-|Q_i\cup P_1\cup P_2|+2$.
Joining $Q_i$ to $P_1$ to $Q_j'$ to $P_2$ gives a red cycle of length $n$ in $K$.

Suppose that for all $i\neq j$, there do not exist two vertex-disjoint paths from $H_i$ to $H_j$ in $G'$. We show that there is a vertex $v$ which separates some $H_a$ from  the others.
\begin{claim}\label{Claim2ConnectedPart}
There is a set $A\subseteq V(G')$, a $2$-connected subgraph $D\subseteq G'$, a vertex $v\in V(D)$, and an index $a\in\{1, \dots, k-1\}$, such that $H_a\subseteq V(D)\subseteq A$, $A\setminus (S\cup \{v\})=H_a\setminus \{v\}$, and $N_{G'}(A-v)\subseteq A$.
\end{claim}
\begin{proof}
Let $D_1, \dots, D_t$ be the maximal $2$-connected subgraphs of $G'$. By maximality we have that $|D_i\cap D_j|\leq 1$ for any $i\neq j$.
By Lemma~\ref{LemmaExpanderConnected},  $H_i$ is $2$-connected for all $i$,  and hence $H_i\subseteq D_j$ for some $j$. By Menger's Theorem we have that each of $D_1, \dots, D_t$ can contain at most one of the sets $H_i$ for $i=1, \dots, k-1$ (since there do not exist two vertex-disjoint paths between $H_i$ and $H_j$ for distinct $i$ and $j$.) 

Let $F$ be an auxiliary graph with $V(F)=\{D_1, \dots, D_t\}$ with $D_iD_j$ an edge whenever $D_i\cap D_j\neq \emptyset$. It is well known that $F$ is a forest (see Proposition 3.11 in~\cite{Diestel}). 
Let $T$ be any subtree of $F$ which contains $H_i$ for some $i$, and let $D_{root}$   be an arbitrary root of $T$. There is  a vertex $D_b\in T$ such  $D_b$ contains $H_a$ for some $a$, but no descendant of $D_b$ contains $H_j$ for any $j\neq a$. Let $D=D_b$.

If  $D_b\neq D_{root}$, then let $D_s$ be the parent of $D_b$ and $v$ the unique vertex in $D_s\cap D_b$. 
Let $A$ be the set consisting of $v$ plus all the vertices  in the connected component of $G'-v$ containing 
$H_a$. 
If $D_b=D_{root}$, then let $A$ be the set consisting of all the vertices  in the connected component of $G'$ containing 
$H_a$ (which is just $\bigcup_{v\in T}D_v$), and  let $v$ be an arbitrary vertex in $D$. 

In both of the above cases, $H_a\subseteq V(D)\subseteq A$ and $N_{G'}(A-v)\subseteq A$ are immediate. To see $A\setminus (S\cup \{v\})=H_a\setminus \{v\}$, recall that $H_1, \dots, H_{k-1}, S$ partitioned $V(G')$ and $D_b$  was chosen so that no descendant of $D_b$ in $T$ contains $H_j$ for any $j\neq a$.
\end{proof}
Let $A$, $v$, and $a$ be as produced by the above lemma.
Notice that $\overline{G[A]}$ is $K_{m+1,m+1}$-free. Indeed given a copy of $K_{m+1,m+1}$ in $\overline{G[A]}$, we have a copy of $K_{m,m}$ in $\overline{G[A]}\setminus \{v\}$. Since from Claim~\ref{Claim2ConnectedPart} there are no edges between this $K_{m,m}$ and $H_{t}\setminus \{v\}$ for $t\neq b$, we obtain a copy of $K_m^k$ in $\overline G$.

Notice that for any $W\subseteq A$ with $|W|\geq 5m+k^{21}+1$, we have 
\begin{align*}
|N_A(W)\cup W|&\geq |N_A(W\setminus (S\cup \{v\}))\cup (W\setminus (S\cup \{v\}))| \\
&= |N_A(W\cap (H_a\setminus \{v\}))\cup (W\cap (H_a\setminus \{v\}))| \\
&\geq n
\end{align*}
To see the last inequality, recall that from Claim~\ref{Claim2ConnectedPart}  we have $A\setminus (S\cup \{v\})=H_a\setminus \{v\}$, $N_{H_a\cup S}(A-v)\subseteq N_{G'}(A-v)\subseteq A$, that  $H_a$ is a $(3,5m,n)$-expander in $G[H_a \cup S]$, and that $|W\cap (H_a\setminus \{v\})|=|W\setminus (S\cup \{v\})|\geq 5m$.

Let $u$ be any neighbour of $v$ in $D$.
Since $|H_a|\geq 10^{52}m$ and $\overline{H_a}$ is $K_{m,m}$-free, $H_a-v-u$ contains a cycle $C$ with $|C|\geq 100m$ (eg. by Theorem~\ref{TheoremRamseyAtLeast}). By $2$-connectedness of $D$ combined with Menger's Theorem, there are two disjoint paths $P_u$ and $P_v$ from $u$ and $v$ respectively to $C$. Joining $P_u$ and $P_v$ to the longer segment of $C$ between $P_u\cap C$ and $P_v\cap C$ gives an $u$ to $v$ path $P$ of length $\geq 50m$. 
 Applying Lemma~\ref{LemmaRamseyConnectGivenVertices} to the graph $G[A]$, the vertices $u$ and $v$, the path $P$, and $m'=5m+k^{21}$,  gives a path of order $n$ from $u$ to $v$ which together with the edge $uv$ forms a cycle of length $n$ in $G$ (and hence a red $C_n$ in $K$.)
\end{proof}

\section{Concluding remarks}
In Theorem~\ref{TheoremCnKmkRamsey} we needed two conditions for $C_n$ to be $K_{m_1, \dots, m_k}$-good---we needed $n\geq 10^{60} m_k$ and  $m_i\geq i^{22}$. 

The first of these conditions ``$n\geq 10^{60} m_k$'' cannot be removed completely (although the constant $10^{60}$ can probably be significantly reduced) as there are constructions showing that $C_n$ is not $K_{m_1, \dots, m_k}$-good for $n\leq m_k$. One family of such constructions is to fix a number $r\in\{1, \dots k\}$ and consider a $2$-edge-colouring of a complete graph on $(k-r)(n-1)+r(m_{r}-1)$ vertices consisting of $(k-r)$ red cliques $C_1, \dots, C_{k-r}$ of size $n-1$ and $r$ red cliques $C_{k-r+1}, \dots, C_{k}$ of size $m_r-1$. For $n\geq m_r$, this construction neither has red $C_n$ nor blue $K_{m_1, \dots, m_k}$---there is no red $C_n$ since all red components have size $\leq n-1$, and there is no blue $K_{m_1, \dots, m_k}$ since the $k$ parts of $K_{m_1, \dots, m_k}$ have to all be contained in different sets $C_1, \dots, C_{k}$, but only $k-r$ of these have size bigger than $m_r$ (and so it is impossible to simultaneously embed the $k-r+1$ parts of $K_{m_1, \dots, m_k}$ of sizes $m_r, m_{r+1}, \dots, m_k$).
This construction shows that for $n\geq m_r$ we have 
$R(C_n, K_{m_1, \dots, m_k})\geq (k-r)(n-1)+r(m_{r}-1)$. For $r=1$, this is exactly (\ref{RamseyLowerBound}).
From this bound we obtain that for $m_r\leq n<  m_r + \frac{m_r-m_1}{r-1}-1$, the cycle $C_n$ is not $K_{m_1, \dots, m_k}$-good. By choosing $r=k$, we see that the bound ``$n\geq 10^{60} m_k$'' in Theorem~\ref{TheoremRamseyAtLeast} cannot be improved significantly beyond ``$n\geq km_k/(k-1)$''.

We conjecture that the second condition  ``$m_i\geq i^{22}$'' in Theorem~\ref{TheoremCnKmkRamsey} can ommited completely. Such a result would in particular show that $C_n$ is $K_m$ good i.e.  it would prove particular cases of Conjecture~\ref{ConjectureCycleComplete}. Because of this it would likely require different proof techniques from the ones used in this paper (for example Nikiforov's ideas from~\cite{Nik} showing that $C_n$ is $K_m$-good for $n\geq 4m+2$ may be helpful).

The gadgets that we use are very similar to \emph{absorbers} introduced by Montgomery in~\cite{Montgomery} during the study of spanning trees in random graphs. An absorber is a graph $A$ with three special vertices $x$, $y$, and $v$ such that $A$ has $x$ to $y$ paths with vertex sets $V(A)$ and $V(A)\setminus \{v\}$. 
While absorbers have a long history, Montgomery's key insight was that they can be found in very sparse graphs with good expansion properties. The graphs in which we need to find gadgets are also very sparse, and structurally the gadgets that we find are a natural generalization of Montgomery's absorbers. However the graphs in which we look for gadgets are even sparser than Montgomery's ones and have weaker expansion properties. Specifically, Montgomery was looking at graphs $G$ in which any small set $S$ satisfies $|N(S)|\geq C |S|\log^4 |G|$, whereas in this paper we consider graphs which only have $|N(S)|\geq C|S|$. The level of expansion at which we find gadgets is optimal up to a constant factor. Since we find gadgets (and as a consequence absorbers) at such a low expansion, our intermediate results are likely to have application in the study of random and pseudorandom graphs.

\vspace{0.15cm}
\noindent
{\bf Acknowledgment.}\, 
Part of this work was done when the second author visited Freie University Berlin.
He would like to thank Humboldt Foundation for a generous support during this visit and 
Freie University for its hospitality and stimulating research environment.

\end{document}